\documentclass[11pt]{article}
\usepackage[utf8]{inputenc}
\usepackage{authblk}
\usepackage{mathtools}
\usepackage[english]{babel}
\usepackage{amsmath}
\usepackage{amsthm}
\usepackage{amsfonts}
\usepackage{amssymb}
\usepackage{xparse}
\usepackage{bbm}
\usepackage{graphicx}
\usepackage{subcaption}
\usepackage[usenames,dvipsnames]{color}
\usepackage{hyperref}
\usepackage{dsfont}
\usepackage{verbatim}
\usepackage{tikz}

\usetikzlibrary{arrows.meta}
\usetikzlibrary{arrows}
\usetikzlibrary{patterns}
\numberwithin{equation}{section}

\usepackage{enumitem}
\usepackage[numeric,initials,nobysame,msc-links,abbrev]{amsrefs}
\renewcommand{\eprint}[1]{\href{https://arxiv.org/abs/#1}{arXiv:#1}}

\newcommand{\pageafter}[1]{#1~pp.}
\BibSpec{article}{%
+{} {\PrintAuthors} {author}
+{,} { \textit} {title}
+{.} { } {part}
+{:} { \textit} {subtitle}
+{,} { \PrintContributions} {contribution}
+{.} { \PrintPartials} {partial}
+{,} { } {journal}
+{} { \textbf} {volume}
+{} { \PrintDatePV} {date}
+{,} { \issuetext} {number}
+{,} { \pageafter} {pages}
+{,} { } {status}
+{,} { \PrintDOI} {doi}
+{,} { available at \eprint} {eprint}
+{} { \parenthesize} {language}
+{} { \PrintTranslation} {translation}
+{;} { \PrintReprint} {reprint}
+{.} { } {note}
+{.} {} {transition}
+{} {\SentenceSpace \PrintReviews} {review}
}
\BibSpec{collection.article}{%
+{} {\PrintAuthors} {author}
+{,} { \textit} {title}
+{.} { } {part}
+{:} { \textit} {subtitle}
+{,} { \PrintContributions} {contribution}
+{,} { \PrintConference} {conference}
+{} {\PrintBook} {book}
+{,} { } {booktitle}
+{,} { \PrintDateB} {date}
+{,} { \pageafter} {pages}
+{,} { } {status}
+{,} { \PrintDOI} {doi}
+{,} { available at \eprint} {eprint}
+{} { \parenthesize} {language}
+{} { \PrintTranslation} {translation}
+{;} { \PrintReprint} {reprint}
+{.} { } {note}
+{.} {} {transition}
+{} {\SentenceSpace \PrintReviews} {review}
}

\usepackage[capitalise]{cleveref}

\newtheorem{theorem}{Theorem}
\crefname{theorem}{Theorem}{Theorems}
\newtheorem{corollary}[theorem]{Corollary}
\crefname{corollary}{Corollary}{Corollaries}
\newtheorem{lemma}[theorem]{Lemma}
\crefname{lemma}{Lemma}{Lemmas}
\newtheorem{proposition}[theorem]{Proposition}
\crefname{proposition}{Proposition}{Propositions}

\crefname{conjecture}{Conjecture}{Conjectures}
\newtheorem{question}[theorem]{Question}
\crefname{question}{Question}{Questions}

\crefname{observation}{Observation}{Observations}
\newtheorem{claim}[theorem]{Claim}
\crefname{claim}{Claim}{Claims}

\crefname{assumption}{Assumption}{Assumptions}
\newtheorem{algorithm}[theorem]{Algorithm}
\crefname{algorithm}{Algorithm}{Algorithms}
\theoremstyle{definition}
\newtheorem{definition}[theorem]{Definition}
\crefname{definition}{Definition}{Definitions}
\newtheorem{remark}[theorem]{Remark}
\crefname{remark}{Remark}{Remarks}
\newtheorem{example}[theorem]{Example}
\crefname{example}{Example}{Examples}

\numberwithin{theorem}{section}

\makeatletter
\pgfdeclarepatternformonly[\LineSpace,\tikz@pattern@color]{my north east lines}{\pgfqpoint{-1pt}{-1pt}}{\pgfqpoint{\LineSpace}{\LineSpace}}{\pgfqpoint{\LineSpace}{\LineSpace}}%
{
    \pgfsetcolor{\tikz@pattern@color}
    \pgfsetlinewidth{0.4pt}
    \pgfpathmoveto{\pgfqpoint{0pt}{0pt}}
    \pgfpathlineto{\pgfqpoint{\LineSpace + 0.1pt}{\LineSpace + 0.1pt}}
    \pgfusepath{stroke}
}
\makeatother

\makeatletter
\pgfdeclarepatternformonly[\LineSpace,\tikz@pattern@color]{my north west lines}{\pgfqpoint{-1pt}{-1pt}}{\pgfqpoint{\LineSpace}{\LineSpace}}{\pgfqpoint{\LineSpace}{\LineSpace}}%
{
    \pgfsetcolor{\tikz@pattern@color}
    \pgfsetlinewidth{0.4pt}
    \pgfpathmoveto{\pgfqpoint{0pt}{\LineSpace}}
    \pgfpathlineto{\pgfqpoint{\LineSpace + 0.1pt}{-0.1pt}}
    \pgfusepath{stroke}
}
\makeatother

\makeatletter
\pgfdeclarepatternformonly[\LineSpace,\tikz@pattern@color]{my horizontal lines}{\pgfqpoint{-1pt}{-1pt}}{\pgfqpoint{\LineSpace}{\LineSpace}}{\pgfqpoint{\LineSpace}{\LineSpace}}%
{
    \pgfsetcolor{\tikz@pattern@color}
    \pgfsetlinewidth{0.4pt}
    \pgfpathmoveto{\pgfqpoint{0pt}{0pt}}
    \pgfpathlineto{\pgfqpoint{\LineSpace + 0.1pt}{0pt}}
    \pgfusepath{stroke}
}
\makeatother

\makeatletter
\pgfdeclarepatternformonly[\LineSpace,\tikz@pattern@color]{my vertical lines}{\pgfqpoint{-1pt}{-1pt}}{\pgfqpoint{\LineSpace}{\LineSpace}}{\pgfqpoint{\LineSpace}{\LineSpace}}%
{
    \pgfsetcolor{\tikz@pattern@color}
    \pgfsetlinewidth{0.4pt}
    \pgfpathmoveto{\pgfqpoint{0pt}{0pt}}
    \pgfpathlineto{\pgfqpoint{0pt}{\LineSpace + 0.1pt}}
    \pgfusepath{stroke}
}
\makeatother

\newdimen\LineSpace
\tikzset{
    line space/.code={\LineSpace=#1},
    line space=5pt
}

\DeclareDocumentCommand \tmix { o } {%
  \IfNoValueTF {#1} {{\ensuremath{t_{\rm mix}}} }{{\ensuremath{t_{\rm mix}^{\mathrm{\scriptstyle{#1}}}}}}%
}
\DeclareDocumentCommand \qc { o } {%
  \IfNoValueTF {#1} {{\ensuremath{q_{\rm c}}} }{{\ensuremath{q_{\rm c}^{\mathrm{\scriptstyle{#1}}}}}}%
}
\DeclareDocumentCommand \qct { o } {%
  \IfNoValueTF {#1} {{\ensuremath{\tilde{q}_{\rm c}}} }{{\ensuremath{\tilde{q}_{\rm c}^{\mathrm{\scriptstyle{#1}}}}}}%
}

\newcommand{\cE}{\mathcal{E}}

\newcommand{\cU}{\mathcal{U}}

\newcommand{\bbE}{{\ensuremath{\mathbb E}} }

\newcommand{\bbH}{{\ensuremath{\mathbb H}} }

\newcommand{\bbN}{{\ensuremath{\mathbb N}} }

\newcommand{\bbP}{{\ensuremath{\mathbb P}} }

\newcommand{\bbR}{{\ensuremath{\mathbb R}} }

\newcommand{\bbZ}{{\ensuremath{\mathbb Z}} }

\newcommand{\bone}{{\ensuremath{\mathbf 1}} }
\newcommand{\bzero}{{\ensuremath{\mathbf 0}} }
\DeclareMathOperator{\diam}{diam}
\DeclareMathOperator{\supp}{supp}

\newcommand{\1}{\mathbbm{1}}
\newcommand{\<}{\langle}
\renewcommand{\>}{\rangle}

\title{\scshape Kinetically constrained models out of equilibrium}

\author{Ivailo Hartarsky}
\author{Fabio Lucio Toninelli}
\affil{Technische Universit\"at Wien, Institut für Stochastik und Wirtschaftsmathematik, Wiedner Hauptstra\ss e 8-10, A-1040, Vienna, Austria
\texttt{\{ivailo.hartarsky,fabio.toninelli\}@tuwien.ac.at}}

\begin{document}

\maketitle

\begin{abstract}
We study the full class of kinetically constrained models in arbitrary dimension and out of equilibrium, in the regime where the density $q$ of facilitating sites in the equilibrium measure (but not necessarily in the initial measure) is close to $1$. For these models, we establish exponential convergence to equilibrium in infinite volume and linear time precutoff in finite volume with appropriate boundary condition. Our results 
are the first out-of-equilibrium results that hold for any model in the so-called critical class, which is covered in its entirety by our treatment. It includes e.g.\ the Fredrickson--Andersen 2-spin facilitated model, in which a site is updated only when at least two neighbouring sites are in the facilitating state. In addition, these results generalise, unify and sometimes simplify several previous works in the field. As byproduct, we recover and generalise exponential tails for the connected component of the origin in the upper invariant trajectory of perturbed cellular automata and in the set of eventually infected sites in subcritical bootstrap percolation models. Our approach goes through the study of cooperative contact processes,  last passage percolation, Toom contours, as well as a very convenient coupling between contact processes and  kinetically constrained models. 
\end{abstract}


\section{Introduction}
Kinetically constrained models (KCM) were introduced in order to study the liquid-glass transition \cite{Fredrickson84} (see \cites{Ritort03,Garrahan11,Cancrini09} for reviews). They are Markov processes featuring a parameter $q\in[0,1]$ tuning the density of facilitating sites. Purposefully, KCM are reversible w.r.t.\ the product Bernoulli measure with parameter $q$. Indeed, it has been proposed that the real-world liquid-glass transition has a purely dynamical origin, that is not reflected in the equilibrium measure. KCM were introduced precisely with the purpose of testing whether glassy behaviour could be explained purely in dynamical terms. Despite having a trivial stationary measure, the degenerate rates of KCM make them very hard to tackle mathematically, by making their dynamics non-attractive, cooperative, heterogeneous, admitting multiple invariant measures, not satisfying coercive inequalities and sometimes featuring ergodicity breaking phase transitions.

In recent years a detailed understanding of KCM at equilibrium has been achieved, especially in two dimensions (see \cite{Hartarsky22phd}*{Chapter 1} for an overview). However, from the physical perspective it is essential to understand their behaviour out of equilibrium, typically after a quench from one temperature to a different one. Rigorous results in this direction are rather limited and will be the subject of a detailed account in \cref{sec:background}. Suffice it to say that with the exception of \cite{Chleboun13}, all out-of-equilibrium results pertain to the class of so-called supercritical modes, for which a finite patch of facilitating sites can trigger relaxation. On the other hand, \cite{Chleboun13} treats subcritical models with an orientation. The main goal of the present work is to deal with critical models. However, our methods work in the greatest possible generality, so we will also cover the other universality classes. Our main results are: the proof that the mixing time of the process in a box of side $n$ and  boundary condition entirely composed of facilitating sites is of order $n$ (see \cref{th:mixing}) and the proof of exponential convergence to equilibrium for the infinite-volume dynamics started from a Bernoulli initial condition {in the ergodic regime} (see \cref{th:convergence}). A comparison with earlier results is given in \cref{sec:background}.

The main novelties of the present work are as follows. Firstly, we set up a general scheme for proving exponential decay for the size of connected clusters of objects in dependent settings (see \cref{sec:Toom}) and showcase two applications of independent interest going far beyond the needs of our main results. We further devise a simple and very robust technique for coupling interacting particle systems, even in the absence of attractiveness (see \cref{sec:coupling}). Moreover, as opposed to previous studies, we develop tools to tackle cooperative models (such that no finite set of facilitating sites is able to propagate), including several useful renormalisation techniques (see \cref{sec:initial:condition,sec:renorm,sec:LPP}). In particular, we provide the first out-of-equilibrium results for any kinetically constrained model in the \emph{critical universality class}.

A major limitation for the study of KCM out of equilibrium is that, like ours, most results work in the perturbative regime of high density of facilitating sites for the equilibrium measure and not necessarily the initial one. Exceptions to this are \cite{Chleboun13}, heavily relying on orientation, and results on the East model \cites{Blondel13b,Chleboun14,Chleboun14,Chleboun16,Couzinie22,Faggionato12,Faggionato13,Ganguly15,Mareche19a}, where a model-specific miracle greatly simplifies the problem.\footnote{Some of the perturbative results on the Fredrickson--Andersen 1-spin facilitated model in one dimension \cites{Blondel13,Blondel19,Ertul22} concern a neighbourhood of $q=1$ with noticeable length. When applied to this model, our treatment also yields reasonable quantitative bounds.} Our study will also be restricted to the perturbative regime for the equilibrium measure, but not the initial one. Nevertheless, it is our hope that, once a robust renormalisation scheme is found for controlling KCM in the non-perturbative regime in terms of the perturbative one, our tools will enable the treatment of the full class of KCM out-of-equilibrium.

\section{Models}
In this section we define our models of interest, KCM, along with several other models, which will play an auxiliary role in the proofs. The reader eager to see the statements of the results in \cref{sec:results} will only need \cref{subsec:families,subsec:KCM}.
\subsection{Update families}
\label{subsec:families}
Let $\|\cdot\|$, $\<\cdot,\cdot\>$ and $d(\cdot,\cdot)$ denote the Euclidean norm, scalar product and distance respectively. An \emph{update family} $\cU$ is a finite non-empty family of finite non-empty subsets of $\bbZ^d\setminus\{0\}$ called \emph{update rules}. We refer to unit vectors 
\[u\in S^{d-1}=\left\{v\in\bbR^d:\|v\|=1\right\}\]
as \emph{directions}. We denote by $\bbH_u=\{x\in\bbZ^d:\<x,u\><0\}$. A direction $u$ is \emph{unstable} for an update family $\cU$, if there exists $U\in\cU$ such that $U\subset\bbH_u$, and \emph{stable} otherwise. An update family is \emph{not trivial subcritical} if it has an unstable direction and \emph{trivial subcritical} otherwise.

While we will not need to distinguish between the other universality classes, it is useful to introduce them for the sake of discussing previous results. We say that an update family is \emph{supercritical}, if there is an open hemisphere consisting only of unstable directions. We say that it is \emph{subcritical} (trivial or non trivial), if every hemisphere contains an open set of stable directions. Finally, the update family is \emph{critical} if it is neither supercritical nor subcritical. We call update families such that there exists $u\in S^{d-1}$ such that $U\subset \bbH_u$ for all $U\in\cU$ \emph{oriented}.

Let us introduce an illustrative example corresponding to the classical \emph{Fredrickson-Andersen $j$-spin facilitated model} (FA-$j$f). Its update family $\cU=\{X\subset\{e_1,\dots,e_d,-e_1,\dots,-e_d\},|X|=j\}$ is given by all $j$-element subsets of the $2d$ nearest neighbours of the origin. One can check that for $j=1$ this family is supercritical, for $j\in\{2,\dots,d\}$ it is critical, while for $j>d$ it is trivial subcritical. The reader may keep in mind the case $j=d$, which is also the most interesting one, throughout the paper. Indeed, all difficulties we face are present for this model.

In the rest of the work, adopting the language of bootstrap percolation and of the contact process, we say that site $x\in\bbZ^d$ is ``infected'' in the configuration $\eta$, if its state at $x$, $\eta_x$, is $1$.

\subsection{Kinetically constrained models}
\label{subsec:KCM}
The $\cU$-KCM is a continuous time Markov process with state space $\Omega=\{0,1\}^{\bbZ^d}$ defined by the following graphical construction (see e.g.\ \cite{Liggett05}*{Section III.6} for background). To each $x\in\bbZ^d$ we attach an independent Poisson point process $P_x$ on $[0,\infty)$ of intensity 1 and uniform random variables $(\Upsilon_x(t))_{t{\in P_x}}$ on $[0,1]$, which are independent and independent of those for other sites. The model has a further parameter $q\in[0,1]$, which we call the \emph{equilibrium density}. We denote by $\eta_x(t)$ the state of site $x\in\bbZ^d$ in the configuration at time $t\in[0,\infty)$. Define the \emph{constraint} at $x\in\bbZ^d$ for a configuration $\eta\in\Omega$ by
\begin{equation}\label{eq:def:cx}
c_x(\eta)=\1_{\exists U\in\cU,\forall u\in U,\eta_{x+u}=1}.\end{equation}
We have 
\begin{equation}
\label{eq:def:KCM}
\eta_x(t)=\begin{cases}
\1_{\Upsilon_x(t)\le q}&t\in P_x\text{ and }c_x(\eta(t-))=1,\\
\eta_x(t-)&\text{otherwise.}
\end{cases}
\end{equation}

Before moving on, let us also give an informal but more intuitive description of the $\cU$-KCM with parameter $q$. Each site $x\in\bbZ^d$ is equipped with a clock which rings at exponentially distributed intervals of time of mean 1 ($P_x$ is the set of clock ring times). When it rings, we verify whether the constraint is satisfied, that is, if there is a completely infected update rule around $x$. If this is the case, we replace the state of $x$ in the configuration by an independent Bernoulli random variable with parameter $q$. Note that in the above definition, we have coupled $\cU$-KCM for all values of $q$ and all initial conditions using the same clock rings and defining the Bernoulli random variables, using the uniform ones, $\Upsilon_x(t)$.

We remark, that $c_x(\eta)$ of \cref{eq:def:cx} does not depend on $\eta_x$, so $c_x(\eta(t-))=c_x(\eta(t))$ for all $t\in P_x$ and $x\in\bbZ^d$. It is easy to see that this implies that the $\cU$-KCM is reversible w.r.t.\ the product Bernoulli measure $\mu_q=\mathrm{Ber}(q)^{\otimes\bbZ^d}$ (see e.g.\ \cite{Liggett05}*{Section IV.2} for background). We emphasise that $\cU$-KCM are not attractive, i.e.\ the natural stochastic order is not preserved by the dynamics (see \cite{Liggett05}*{Section~III.2} for background). We refer to sites in state $1$ as \emph{infected} and sites in state $0$ as \emph{healthy}. Hence, the constraint asks for the presence of suitably arranged infections around the site we are trying to update.

One defines $\qc[KCM]$ as the infimum of all $q\in[0,1]$ such that 0 is a simple eigenvalue of the generator of the $\cU$-KCM with parameter $q$. That is, $\qc[KCM]$  is the critical parameter for ergodicity. It is known by \cite{Cancrini08}*{Proposition~2.5} and \cite{Balister24}*{Corollary~1.6 and Theorem~7.1} that $\qc[KCM]>0$ if and only if $\cU$ is subcritical and $\qc[KCM]=1$ if and only if $\cU$ is trivial subcritical. We will also need the critical parameter of the spectral gap of the generator of the $\cU$-KCM (see \cite{Cancrini08}*{Section 2} for background, but understanding this definition is not essential to the present work):
\[\qct[KCM]=\inf\left\{q>0:\mathrm{gap}>0\right\}.\]
It is believed that $\qct[KCM]=\qc[KCM]$ for all update families, but this has only been shown in some cases (see \cref{rem:qc:qct}).

Fix $\Lambda\subset\bbZ^d$. For any $\omega\in\Omega$ we denote by $\omega_\Lambda\in\Omega_\Lambda=\{0,1\}^\Lambda$ the restriction of $\omega$ to $\Lambda$. We define the $\cU$-KCM $\eta$ on $\Lambda$ with boundary condition $\tau\in\Omega_{\bbZ^d\setminus\Lambda}$ by setting the configuration equal to $\tau$ outside $\Lambda$ at all times. We denote the fully infected (resp.\ healthy) configuration by $\bone$ (resp.\ $\bzero$).

We next introduce the \emph{mixing time} of the $\cU$-KCM  on a finite set $\Lambda$ with some boundary condition $\tau\in\Omega_{\bbZ^d\setminus\Lambda}$ (see \cite{Levin09} for background). Given $\delta\in(0,1)$, we define
\begin{equation}
\label{eq:def:tmix}
\tmix(\delta)=\inf\left\{t\ge 0: {\max_{\rho\in\Omega_\Lambda}}d_{\mathrm{TV}}\left(\bbP\left(\eta^{\rho}(t)\in\cdot\right),\mu_q\right)\le \delta\right\}\in(0,\infty],
\end{equation}
where $d_{\mathrm{TV}}(\mu,\nu)=\sup_{A}(\mu(A)-\nu(A))$ with the supremum running over all events $A$ for the arbitrary probability measures $\mu$ and $\nu$ and $\eta^{\rho}$ is the $\cU$-KCM with initial condition $\rho$.

\subsection{Contact processes}
\label{subsec:CP}
The $\cU$-contact process (CP) is defined by the same graphical construction as the $\cU$-KCM and the same constraint as in \cref{eq:def:cx}. However, we set
\begin{equation}
\label{eq:def:CP}
\zeta_x(t)=\begin{cases}
1&t\in P_x, c_x(\zeta(t-))=1, \Upsilon_x(t)\le q,\\
0&t\in P_x, \Upsilon_x(t)>q,\\
\zeta_x(t-)& \text{otherwise},
\end{cases}
\end{equation} instead of \cref{eq:def:KCM}. That is, the constraint $c_x(\zeta(t-))=1$ is no longer required to be satisfied in order to update the configuration at site $x$ to the value $0$. We define the $\cU$-CP in finite volume with a boundary condition analogously to what was done for KCM in \cref{subsec:KCM}. We emphasise, that we use the same Poisson processes (clock rings) and uniform random variables, so that now $\cU$-KCM and $\cU$-CP for all initial conditions and parameters $q\in[0,1]$ are coupled on the same probability space.

Contrary to KCM, CP are attractive (see \cite{Liggett05}*{Section III.2} for background), we may therefore define its upper invariant measure $\bar\nu$ as the $t\to\infty$ limit in law of $\zeta^\bone(t)$, the $\cU$-CP with initial condition $\bone$. We then define the critical point
\begin{equation}
\label{eq:def:qcCP}\qc[CP]=\inf\left\{q>0:\bar\nu\neq\delta_{\bzero}\right\}.\end{equation}
In other words, $\qc[CP]$ is the critical parameter, above which $\cU$-CP has multiple invariant measures.

\begin{remark}
\label{rem:qcCP}
It is known that $\qc[CP]<1$ if and only if $\cU$ is not trivial subcritical \cite{Gray99}*{Corollary 18.3.2} (the easier ``only if'' direction is contained in the proof of \cite{Gray99}, but also follows from \cite{Balister24}*{Lemma 7.3}). On the other hand, a classical comparison with a branching process shows that $\qc[CP]>0$ for any $\cU$ (see e.g.\ \cite{Liggett99}*{Section I.1}).
\end{remark}
We also extend \cref{eq:def:tmix} to $\cU$-CP without change.

\subsection{Cellular automata with death}
\label{subsec:death}
A \emph{cellular automaton} (CA or $\phi$-CA when we want to emphasise the dependence on the map $\phi$) is specified by a map $\phi:\Omega\to\{0,1\}$ depending only on finitely many sites. Given an initial condition $\omega(0)\in\Omega$, we inductively define for all $t\ge 1, x\in \mathbb Z^d$
\[\omega_x(t)=\phi\left(\omega_{\cdot-x}(t-1)\right).\]
In other words, the map $\phi$ is applied at each site simultaneously in a translation invariant way. The CA is said to be an \emph{eroder}, if for all finite $A\subset \bbZ^d$ there exists $T(A)<\infty$ such that $\omega(0)=\1_{\bbZ^d\setminus A}$ implies $\omega(T(A))=\bone$, that is, finite sets of 0s become extinct after a finite time. The CA is \emph{attractive}, if $\phi$ is non-decreasing for the pointwise partial order on $\Omega$. That is, if $\omega_x\ge\omega'_x$ for some $\omega,\omega'\in\Omega$ and all $x\in\bbZ^d$, then $\phi(\omega)\ge \phi(\omega')$.

Given a cellular automaton $\phi$, we further consider its version with \emph{death} as follows. For $x\in\bbZ^d$ and $t\in\bbN$, let $\xi_{x,t}$ be i.i.d.\ Bernoulli variables with parameter $\delta\in[0,1]$. Then the automaton with map $\phi$ and $\delta$ death is defined by 
\begin{equation}
\label{eq:def:death}\omega_x(t)=\begin{cases}\phi(\omega_{\cdot-x}(t-1))&\xi_{x,t}=0,\\
0&\xi_{x,t}=1,
\end{cases}\end{equation}
for all $t\ge 1$ starting from a given initial condition $\omega(0)\in\Omega$. That is, at each space-time point, we apply the map $\phi$, as in the $\phi$-CA, with probability $1-\delta$ and we change the state to healthy with probability $\delta$. 

For attractive $\phi$, we further denote by $\bar\nu$ the \emph{upper invariant measure}---the limit of the law of the cellular automaton with $\delta$ death and initial condition $\bone$, $\omega^\bone(t)$, as $t\to\infty$. One then defines its stability threshold
\[\delta_{\mathrm{c}}=\sup\left\{\delta\in[0,1]:{\bar\nu\neq\delta_{\bzero}}\right\},\]
where $\delta_\bzero$ is the Dirac measure on $\bzero$. This threshold captures the point up to which death is not strong enough to extinguish infection, if we start from the completely infected state. It is a classical result of Toom \cite{Toom80} that $\delta_{\mathrm{c}}>0$ if and only if the $\phi$-CA is an eroder.

\subsection{Bootstrap percolation}
\label{subsec:BP}
The $\cU$-bootstrap percolation (BP) is the particular CA whose map $\phi$ is defined via
\begin{equation}
\label{eq:def:BP}
\phi(\omega)=\max(\omega_0,c_x(\omega))
\end{equation}
for any $\omega\in\Omega$ with $c_x$ from \cref{eq:def:cx}. In other words, infected sites remain infected, while healthy ones become infected once their constraint is satisfied (i.e.\ there are enough infections around in the sense of $\cU$). One commonly considers
\[\qc[BP]=\inf\left\{q\in[0,1]:\lim_{t\to\infty}\bbP\left(\omega^{\mu_q}_0(t)=0\right)=0\right\},\]
where $\omega^{\mu_q}$ is the $\cU$-BP with initial condition distributed according to $\mu_q$ (there is no other randomness involved). This threshold reflects at which point initial infections with density $q$ are sufficient to almost surely infect the entire lattice. It is known \cite{Cancrini08}*{Proposition~2.5} that in fact for the same update family $\cU$ we have $\qc[KCM]=\qc[BP]$.

Since $\cU$-BP is a CA, one may consider its version with death. That is, for $\delta\ge 0$, $\cU$-BP with $\delta$ death is the $\phi$-CA with $\delta$ death with $\phi$ from \cref{eq:def:BP}. Recalling \cref{subsec:death}, this means that at each time step infected vertices become healthy with probability $\delta$ and stay infected with probability $1-\delta$; healthy vertices whose constraint is satisfied become infected with probability $1-\delta$; healthy sites whose constraint is not satisfied remain healthy with probability 1.

\subsection{Last passage percolation}
\label{subsec:LPP}
Given an update rule $U\subset\bbH_u$ for some $u\in S^{d-1}$, we define the $U$-last passage percolation (LPP) on $\Lambda=\{1,\dots,n\}^d$ as follows. Endow each $x\in\Lambda$ with an i.i.d.\ exponentially distributed random variable $T(x)$ with mean $1$. For each $x\not\in\Lambda$ set $s_x=0$. For every $x\in\Lambda$ inductively define the \emph{$U$-LPP time of $x$} by 
\begin{equation}
\label{eq:def:LPP}
s_x=T(x)+\max_{y\in U}s_{x+y}.
\end{equation}
Indeed, this is possible, because $U$ is contained in an open half-plane. Another way to view $U$-LPP is the following. Sites in $\Lambda$ are initially healthy and those in $\bbZ^d\setminus\Lambda$ are infected. When all neighbours of a site (in the sense of $U$) are infected, it becomes infected at rate 1 and never heals afterwards.

With this representation in mind, it is not hard to check that the set of vertices $x\in\bbZ^d$ where $s_x\le t$ for the $U$-LPP coincides with the set of infected vertices at time $t$ in the configuration of the $\{U\}$-KCM in the box $\Lambda$, with entirely infected boundary condition, healthy initial condition and parameter $q=1$. Note also that if $U={\{0,-1\}^{d}\setminus\{0\}}$, the $U$-LPP on $\Lambda$ coincides with the standard $\{-e_1,\dots,-e_d\}$-LPP on $\Lambda$. Indeed, e.g.\ $s_{x-e_1-e_2}\le s_{x-e_1}\le s_x$ for any $x\in\Lambda$ for $\{-e_1,\dots,-e_{d}\}$-LPP, so the maximum over $U$ in \cref{eq:def:LPP} coincides with the maximum over $\{-e_1,\dots,-e_{d}\}\subset U$.
\section{Results}
\label{sec:results}
We are now ready to state our main results. The first one concerns the mixing time of KCM or CP in finite volume with infected boundary.
\begin{theorem}[Linear mixing]
\label{th:mixing}
Let $\cU$ be an update family, which is not trivial subcritical. There exists $\varepsilon=\varepsilon(\cU)>0$ such that for all $q\in[1-\varepsilon,1]$ the $\cU$-KCM on $\{1,\dots,n\}^d$ with infected boundary condition exhibits precutoff in linear time: there exists $C=C(\cU)>0$ such that for all $\delta\in(0,1)$ and $n$ large enough depending on $\delta$,
\begin{equation}
\label{eq:UBLBtmix}
n/C\le \tmix(\delta)\le Cn.    
\end{equation}
The same holds for the $\cU$-CP.
\end{theorem}

\begin{remark}
\label{rem:optimality}Let us note that all the conditions above are essential. Indeed, trivial subcritical models are simply not ergodic, as they admit finite healthy regions which cannot change \cite{Balister24}*{Lemma 7.3}. Moreover, there exist non-trivial subcritical  models with $\qc[KCM]$ arbitrarily close to 1 \cite{Hartarsky21}*{Proposition 7.1}, so one cannot hope $\varepsilon$ to be independent of $\cU$. Finally, one cannot change the boundary condition to healthy or periodic (restricting to the ergodic component), since it is known that even 2-neighbour bootstrap percolation with these boundary conditions may have quadratic infection time \cite{Benevides15}, and the $\cU$-BP infection time is a lower bound on the $\cU$-KCM one (see \cite{Martinelli19}*{Lemma 4.3}).
\end{remark}

Our second main result  establishes that, in infinite volume, $\cU$-KCM converge exponentially fast to their equilibrium measure.
\begin{theorem}[Exponential convergence]
\label{th:convergence}
Let $\cU$ be any update family and $\alpha>0$. Then,
 there exist $\varepsilon\in(0,1)$ and $c>0$, such that for any  $p\in[\qct[KCM]{+\alpha},1]$ and $q\in[1-{\varepsilon},1]$ the following holds. Let $(\eta^{\mu_p}(t))_{t\ge 0}$ be the infinite volume $\cU$-KCM with initial distribution\footnote{The initial condition is assumed to be product mostly for simplicity. It will be clear from the proof that e.g.\ any initial condition stochastically dominating $\mu_p$ would do.} $\mu_p$ and parameter $q$. Then for all local functions $f:\Omega\to \bbR$ and $t\ge 0$
\begin{equation}
\label{eq:main:annealed}
\left|\bbE\left[f(\eta^{\mu_p}(t))\right]-\mu_q(f)\right|\le e^{-ct}\|f\|_\infty\cdot|\supp f|/c,\end{equation}
where $\supp f$ is the set of sites on whose state the value of $f$ depends.
\end{theorem}
As in \cref{rem:optimality}, one cannot hope for $\varepsilon$ independent of $\cU$. However, one should expect both \cref{th:mixing,th:convergence} to hold for any $q>\qc[KCM]+\alpha$.
\begin{remark}
\label{rem:qc:qct}
Since $\qct[KCM]$ appears directly in \cref{th:convergence}, let us mention that  for supercritical and critical models it is known that $\qct[KCM]=\qc[KCM]=0$ \cites{Hartarsky21,BalisterNaNb}.\footnote{\label{foot:qctBP}
For supercritical and critical models \cite{BalisterNaNb} gives a stretched exponential decay of the tail of the infection time of the origin for any $q>0$, so $\qct[KCM]=0$ for non-subcritical update families by \cite{Hartarsky21}*{Theorem~3.7}. Note that, while \cite{Hartarsky21} is formulated in two dimensions, the parts we will use do generalise rather straightforwardly to higher dimensions.} 
For trivial subcritical update families we have $\qct[KCM]=\qc[KCM]=1$ \cites{Cancrini08,Balister24}, so that \cref{th:convergence} is empty, as it should. For subcritical non-trivial families it is known that $0<\qc[KCM]\le \qct[KCM]<1$ \cites{Cancrini08,Balister24,Hartarsky21} and the second inequality is believed to be an equality, but this is an important open problem. Equality has been shown for oriented update families \cite{Hartarsky22sharpness} and the general case was reduced to a related question involving only $\cU$-BP in \cite{Hartarsky21}. Furthermore, \cref{th:convergence} cannot hold for $p<\qc[KCM]=\qc[BP]$, since then a.s.\ there are sites whose state remains 0 forever. Hence, if the conjecture $\qc[KCM]=\qct[KCM]$ holds, the range of values for $p$ in the theorem is the best possible.
\end{remark}
Finally, let us mention that in the course of the proof of our main results we will derive consequences on exponential decay in space-time for eroder attractive cellular automata with death, and on space exponential decay for the set of sites eventually reaching state 1 in $\cU$-BP for subcritical $\cU$ with initial condition $\mu_q$. The reader interested in these developments (\cref{cor:BP,cor:Toom}) can directly refer to \cref{sec:Toom}, which can be read independently of the remainder of the paper.

\section{Background}
\label{sec:background}
Before turning to the proof of \cref{th:convergence,th:mixing}, let us discuss previous work on KCM out of equilibrium. We start by mentioning that results of a different kind, concerning graphs whose size diverges jointly with the parameter $q$ tending to $0$,  can be found in \cites{Chleboun14,Pillai17,Pillai19,Hartarsky22CBSEP}, while large deviations in trajectory space have been studied in \cite{Bodineau12}.

Along the lines of our work, much more has been done, mostly for supercritical models, especially the $\{\{-e_1\},\dots,\{-e_d\}\}$-KCM called the East model and FA-1f, that is, the $\{\{-e_1\},\dots,\{-e_d\},\{e_1\},\dots,\{e_d\}\}$-KCM. In all cases roughly the same route has been followed, to the extent possible, along the following steps in that order, each one relying on the previous one.
\begin{enumerate}[label={Step \arabic*.}, ref={Step \arabic*}]
    \item\label{step:convergence} \Cref{th:convergence}, possibly with a weaker stretched exponential decay.
    \item\label{step:precutoff} Positive speed of the infection front and the corresponding precutoff, that is, \cref{th:mixing}.
    \item\label{step:front} Ergodicity of the process seen from the front and law of large numbers for the front position.
    \item\label{step:cutoff} CLT for the front position and cutoff.
\end{enumerate}
\ref{step:convergence} was performed first for the East model in $d=1$ in \cite{Cancrini10}. Like all results on this model, that relies on orientation and further favourable properties, allowing the results to hold for all $q>0$. Certain qualitative convergence results for a model closely related to FA-1f were obtained in \cite{Sudbury97a} (also see there references therein). For FA-1f, \ref{step:convergence} was done with stretched exponential decay for FA-1f for $q>1/2$ in \cite{Blondel13} (pure exponential in $d=1$). For the East model in $d>1$ stretched exponential decay was proved in \cite{Chleboun15}. Convergence for FA-1f was improved to pure exponential in \cite{Mountford19} for $q$ large enough. For the East model in $d>1$ the same was done in \cite{Mareche19a}. The exponential decay was then generalised to all supercritical models in any dimension for $q$ large enough in \cite{Mareche20}, thus including the ones of \cites{Mountford19,Blondel13}, as well as \cites{Cancrini10,Chleboun15,Mareche19a} up to the restriction on $q$. The most general of the above results, \cite{Mareche20}, is contained in \cref{th:convergence}. In summary, before the present work, \cref{th:convergence} was known only for supercritical KCM.

Turning to \ref{step:precutoff}, \cref{th:mixing} was proved  for the East model for any $q>0$ in \cite{Chleboun15}. We believe that even for FA-1f in dimension $d>1$ \cref{th:mixing} is new. However, there is another important work in the direction of \cref{th:mixing}. Namely, in \cite{Chleboun13} this result was proved with a weaker upper bound of order $n\log n$, assuming that the update family is oriented. It should be noted that orientation rules out the possibility for the model to be critical and is a very convenient feature for the analysis, as we will see. On the other hand, the approach of \cite{Chleboun13} has the major advantage of working for any $q>\qc[KCM]$ owing to \cite{Hartarsky22sharpness}*{Theorem 1.6} and \cite{Hartarsky21}*{Theorem 3.7}.

Moving on to \ref{step:front}, in $d=1$ it has been established that the front has a well-defined speed and that the law of the configuration behind the front converges to a limit for large times. This is done for the East model in \cite{Blondel13b} for any $q>0$ and in \cite{Blondel19} for FA-1f for $q$ large enough. In \cites{Chleboun16,Couzinie22} the East model was studied in $d>1$ with the aim to examine the limit shape of the set of updated sites starting from a single infection. Results are still far from establishing that such a limit shape actually exists, but some control on the speed of the front in different directions is obtained.

Finally, \ref{step:cutoff} was achieved in $d=1$ for the East model in \cite{Ganguly15} (see \cites{Faggionato12,Faggionato13} for further results about the $d=1$ East process out of equilibrium) and for FA-1f at $q$ close to 1 in \cite{Ertul22}. This was also obtained for the East model in higher dimensions for a particular domain and boundary condition in \cite{Couzinie22}. We should also mention that in \cite{Lacoin14}*{Theorem 2.4} \ref{step:cutoff} was performed for FA-2f in $d=2$ in the somewhat degenerate case $q=1$, which also coincides with the zero-temperature Ising model with appropriate external field. This can also be viewed as a continuous time version of 2-neighbour BP or a non-oriented LPP.

While it would be extremely interesting to see \ref{step:front} and \ref{step:cutoff} established for general update families, this seems rather remote, given that even the $d$-dimensional East model for $d>1$ has not been handled at that level yet. As it is the case for the 1-dimensional East model and FA-1f, we expect that our \cref{th:convergence,th:mixing} and the tools developed to prove them will play an important role in attacking these questions.

\section{Outline of the proof}
Let us start by sketching the proof of \cref{th:convergence}, which is slightly simpler than the one of \cref{th:mixing}. The proof is composed of several steps corresponding to  \cref{sec:initial:condition,sec:coupling,sec:renorm,sec:Toom}, which are put together in \cref{sec:convergence}. 

The first step (see \cref{sec:initial:condition}) consists in ``warming up'' the initial condition. That is, we improve our initial condition $\mu_p$ with $p>\qct[KCM]$ to one with high density of infections. This is achieved via a renormalisation drawing on \cite{Hartarsky21}. It roughly says that, since $p>\qct[KCM]$, the probability that a site does not become infected within time $t$ in the $\cU$-BP dynamics with the same initial condition decays exponentially with $t$. Since the parameter $q$ of our $\cU$-KCM is close to 1, the same holds for it up to a large enough time, since the $\cU$-KCM essentially reduces to the $\cU$-BP in the absence of recovery events. Hence, looking at a renormalised lattice, we may assume that the initial condition of our dynamics is product with high density of infections.

The second step (see \cref{sec:coupling}) is to reduce the study of the $\cU$-KCM to the CP with update family consisting of a single rule $U_0$ that is oriented. While there is a standard monotone comparison (see \cref{claim:comparison}) that guarantees that all infections of the $\{U_0\}$-CP are infections in the $\cU$-KCM, we need to go further. Namely, we establish that studying a certain set that depends not only on the configuration, but also on the history of the CP, we are able to deduce that the $\cU$-KCM not only has lots of infections, but has actually coupled for all initial conditions larger than the one of the CP. 

Hence, we have reduced our problem to one about the $\{U_0\}$-CP with parameter $q_0$ close to 1 and initial condition close to $\bone$. In doing so, we have lost the reversibility and the product invariant measure of the KCM, but we have gained attractiveness and orientation for the CP. Nevertheless, we are not done yet, because the $\{U_0\}$-CP is by far not as simple as the classical CP, as its dynamics is still cooperative, because $U_0$ may contain more than one vertex.

The third step (see \cref{sec:renorm}) is a further renormalisation generalising and somewhat simplifying the one of \cites{Mareche20}, itself stemming from \cite{Mountford19}. It transforms the $\{U_0\}$-CP into a $\cU_0$-BP with (little) death, where $\cU_0=\{\{0,-1\}^d\setminus \{0\}\}$. To achieve this, we tessellate space-time into large boxes of carefully chosen geometry. Roughly speaking, we ensure that if all neighbours of a box in the directions given by $\cU_0$ are fully infected, then infection propagates with high probability to the box of interest. Moreover, an infected box remains such at the next time step with high probability. If either of these high probability events fails, we view that as a death in the renormalised BP process. We have thus made our process even simpler, as it now evolves in discrete time and no longer depends on the original update family $\cU$, while remaining cooperative.

The fourth step (see \cref{sec:Toom}) is to show that the $\cU_0$-BP is exponentially unlikely to have large space-time clusters of healthy sites. This can be traced back through the previous steps to the CP and then the KCM to yield \cref{th:convergence} (see \cref{sec:convergence}). In fact, we prove this in general for any CA with death which is an eroder (recall \cref{subsec:death}), by developing a novel and very general scheme for leveraging exponential bounds {on the probability of occurrence of objects rooted at a given point} to exponential bounds on clusters {of such objects} and applying this to Toom contours \cite{Toom80}, as recently revisited in \cite{Swart22}. Alternatively, one could employ a multi-scale renormalisation argument, as in \cite{Balister24} transported to the setting of CA with death via \cite{Hartarsky22sharpness}, but this would degrade \cref{th:convergence} to stretched exponential convergence at best.

Finally, we turn to the proof of \cref{th:mixing}, which still relies on all of the above, but requires a substitute for the initial ``warming up'' step above, since we need to deal with arbitrary initial conditions, including $\bzero$. This alternative first step (see \cref{sec:LPP}) consists in yet another renormalisation, this time from the $\{U_0\}$-CP to the standard LPP. We show that after the LPP time at a renormalised site, the corresponding box is coupled for the CP. To do this,  somewhat surprisingly, we look at times when sites become healthy in the CP. Using the orientation of the $\{U_0\}$-CP, we have that, once all sites on whose state the constraint at a given site $v$ depends on are coupled, it remains to wait for a single update at $v$ to the state 0, in order to couple the state at $v$, too. Although these {updates} are rare ($q$ is close to 1), {it still takes a time of order 1 to couple the entire box corresponding to a renormalised site, given that the ones it depends on (in the LPP sense) are already coupled.} Thus, to ensure the CP is coupled on a box $\Lambda$ of size $n$, it suffices to wait until the LPP on a (renormalised) box of size of order $n$ reaches all sites. The latter is known to be linear from \cite{Greenberg09}. 

Hence, after a linear time the CP has reached its equilibrium distribution in the box with $\bone$ boundary condition. By attractiveness of CP, this distribution stochastically dominates the restriction to the box $\Lambda$ of the infinite volume upper stationary measure. Moreover, the CP lower bounds the KCM with any initial condition, so, after this burn-in time, we may perform the same procedure as for the proof of \cref{th:convergence}. Namely, we exploit the relation between $\{U_0\}$-CP and $\cU$-KCM, renormalise the former to $\cU_0$-BP with death and, finally, use that the latter has exponentially small probability to have large healthy space-time clusters. Using this exponential bound, we get that it suffices to wait for a time of order $\log n$ to ensure that the $\cU$-KCM has coupled with high probability, once we have waited for the initial linear burn-in time needed for the LPP to reach all sites in the renormalised version of $\Lambda$.

\section{Warming up the initial condition}
\label{sec:initial:condition}
In this section we start by showing that if the initial condition is  $\mu_p$, with $1-p<1-\qct[KCM]$ possibly much larger than the equilibrium  density $1-q$  of healthy sites, after a sufficiently large but finite time the law of the process dominates a renormalised Bernoulli measure with \emph{large} infection density{, but still not the equilibrium one}.

Let us note that the vectors $v'_i$ in the next lemma are arbitrary at this point. The reader is encouraged to think of them as the canonical basis of $\bbR^d$, while a more convenient choice will appear in \cref{sec:renorm}.
\begin{lemma}
\label{lem:initial:cond}
Let $\cU$ be an update family, $\alpha>0$, $\varepsilon_0>0$ and let $v_1',\dots,v_d'\in\bbZ^d$ be linearly independent. There exists $R_0\in \{1,2,\dots\}$ such that for any $R\ge R_0$ there exists $T_0$ such that for any $T\ge T_0$ there exists $\varepsilon_1>0$ such that for any $q\in[1-\varepsilon_1,1]$ the following holds. Set 
\[\hat B=\sum_{i=1}^d(v'_i[0,R))=\left\{\sum_{i=1}^da_iv'_i:(a_i)_{i=1}^d\in[0,R)^d\right\}\subset \bbR^d\] and $\hat B_x=\hat B+\sum_{i=1}^d Rx_iv'_i$ for every $x\in\bbZ^d$. We can couple all $\cU$-KCM $\eta^{\mu_{p'}}$ with parameter $q$ and initial conditions $\mu_{p'}$ for $p'\ge \qct[KCM]+\alpha$ together with $\xi\sim\mu_{1-\varepsilon_0}$ so that for every $x\in\bbZ^d$, \[\xi_x=1\Rightarrow\;\forall p'\ge \qct[KCM]+\alpha,\eta^{\mu_{p'}}_{\hat B_x}(T)=\bone_{\hat B_x}.\]
\end{lemma}
\begin{proof}
Fix $\alpha>0$ and set $p=\qct[KCM]+\alpha$. From \cite{Hartarsky21}*{Theorems 3.5 and 3.7} we have that there exists $c=c(\alpha)\in(0,\infty)$  such that the $\cU$-BP (without death) $\omega$ satisfies $\bbP(\omega_0^{\mu_p}(t)=0)\le e^{-ct}$ for every integer $t\ge 0$. Hence, for any fixed $\varepsilon'>0$, we can choose $R>0$ large enough depending on $c$ and $\varepsilon'$ such that $\sqrt{R}\in\bbN$ and 
\[\bbP\left(\omega_{\hat B}^{\mu_p}\left(\sqrt{R}\right)=\bone_{\hat B}\right)\ge 1-\varepsilon'.\]
Since $R$ is large enough, the above event only depends on $\omega^{\mu_p}(0)$ restricted to $\bigcup_{z\in\{-1,0,1\}^d}\hat B_z$.

Recall that $P_x$ is the Poisson process of clock ring times associated to $x\in\bbZ^d$ from \cref{subsec:KCM} used to couple the $\cU$-KCM with all initial conditions. We further assume initial conditions distributed according to $\mu_{p'}$ for $p'\ge p$ to be coupled in a monotone way. Let us now choose $T>0$ large enough depending on $R$ in such a way that with probability at least $1-\varepsilon'$ for each $i\in\{0,\dots,\sqrt{R}-1\}$ and $x\in \bigcup_{z\in\{-1,0,1\}^d}\hat B_{z}$ we have $P_x\cap (iT/\sqrt R,(i+1)T/\sqrt R)\neq \varnothing$. That is, in each of these $\sqrt R$ intervals of time the clock of each site rings. Finally, if $q$ is close enough to 1 depending also on $T$, we get that with probability at least $1-\varepsilon'$, we have $\Upsilon_x(t)\le q$ for all $t\in P_x\cap [0,T)$ and $x\in\bigcup_{z\in\{-1,0,1\}^d}\hat B_z$. That is, at each clock ring, we attempt to infect the corresponding site.

We claim that if all three events above occur  for some $x\in\bbZ^d$, that is, $\omega^{\mu_p}_{\hat B_x}(\sqrt{R})=\bone_{\hat B_x}$, $P_y\cap(iT/\sqrt{R},(i+1)T/\sqrt R)\neq \varnothing$ and $\Upsilon_y(t)\le q$, for all $y\in{\bigcup_{z\in x+\{-1,0,1\}^d}\hat B_z}$, $i\in\{0,\dots,\sqrt{R}-1\}$ and $t\in P_y\cap[0,T)$, then \begin{equation}
\label{eq:claimeta}
    \eta_{\hat B_x}^{\mu_{p'}}(T)=\bone_{\hat B_x} \text{ for any } p'\ge p.
\end{equation} Since $\omega^{\mu_p}_{\hat B_x}(\sqrt R)=\bone_{\hat B_x}$ and $R$ is large (hence much larger than $\sqrt R$), there exists $X$ with $\hat B_x\subset X\subset \bigcup_{z\in\{-1,0,1\}^d}\hat B_{x+z}$ such that the $\cU$-BP process $\omega^{\mu_p,X}$ on $X$ with boundary condition $\bzero_{\bbZ^d\setminus X}$ and initial state $\omega_X^{\mu_p}(0)$ satisfies $\omega^{\mu_p,X}(\sqrt{R})=\bone_X$.
\footnote{To see this, one may remove the infections at distance more than $C\sqrt{R}$ from $\hat B_x$ for some large constant $C>0$, since they do not reach $X$ by time $\sqrt R$, and take $X$ to be the set of all sites infected by the remaining ones up to time $\sqrt R$.} Then we can prove by induction on $i\in\{0,\dots,\sqrt{R}\}$ that for all $y\in {X}$ we have $\eta_{y}^{\mu_{p'}}(iT/\sqrt{R})\ge \omega_y^{\mu_p,X}(i)$. 

The base is the monotone coupling of the initial conditions. For the induction step, observe that by assumption, for the KCM, no attempt is made to change any state to 0, but at least one attempt is made to change each site to 1. Since the constraint $c_y$ is non-decreasing in the configuration and the $\eta_X^{\mu_{p'}}$ process is non-decreasing in time, if for some $y\in X$ and $i\in\{0,\dots,\sqrt{R}-1\}$ we have $c_y(\omega^{\mu_p,X}(i))=1$, then for any $t\in(iT/\sqrt{R},(i+1)T/\sqrt{R})$ we have $c_y(\eta^{\mu_{p'}}(t))=1$. Applying this to some $t\in P_y\cap(iT/\sqrt{R},(i+1)T/\sqrt{R})\neq\varnothing$, we obtain the induction step: 
\begin{align*}
\eta^{\mu_{p'}}_y\left((i+1)T/\sqrt{R}\right)&{}\ge\eta^{\mu_{p'}}_y(t)=c_y\left(\eta^{\mu_{p'}}(t)\right)\vee \eta^{\mu_{p'}}_y(t-)\\
&{}\ge c_y\left(\omega^{\mu_p,X}(i)\right)\vee \eta^{\mu_{p'}}_y\left(iT/\sqrt{R}\right)\\&{}\ge c_y\left(\omega^{\mu_p,X}(i)\right)\vee \omega^{\mu_p,X}_y(i)=\omega^{\mu_p,X}(i+1).\end{align*}
Hence, the claimed \cref{eq:claimeta} follows by taking $i=\sqrt {R}$.

Thus, for every $x\in\bbZ^d$, on an event $E_x$ of probability at least $1-3\varepsilon'$, the $\cU$-KCM with initial condition $\mu_{p'}$ coupled in a monotone way satisfy
\cref{eq:claimeta}.
Moreover, by construction the events $E_x$ are 1-dependent in terms of $x$, so by the Liggett--Schonmann--Stacey Theorem \cite{Liggett97}, the set of $x$ such that $E_x$ is realised stochastically dominates an i.i.d.\ configuration with parameter at least $1-\varepsilon_0$ such that $\varepsilon_0\to 0$ if $\varepsilon'\to 0$.
\end{proof}

\section{Coupling KCM and CP}
\label{sec:coupling}
In this section we examine the coupling between KCM and CP. Let $\cU$ be an update family which is not trivial subcritical and let 
\begin{equation}
\label{eq:def:norm:U}
\|\cU\|=\max_{U\in\cU,x\in U}\|x\|.
\end{equation}
Fix $U_0\in\cU$  such that $U_0\subset \bbH_u$ for some $u\in S^{d-1}$ (this is possible since the update family $\mathcal U$ is assumed not to be a trivial subcritical one). Let us fix a domain $\Lambda\subset \bbZ^d$ and boundary condition $\tau\in\Omega_{\bbZ^d\setminus\Lambda}$. Fix two parameters $0\le q_0\le q\le 1$. Fix an initial condition $\xi\in\Omega_{\Lambda}$ and denote by $\zeta$ the $\{U_0\}$-CP on $\Lambda$ with boundary condition $\tau$, initial condition $\xi$ and parameter $q_0$. Recall from \cref{subsec:CP,subsec:KCM} that KCM and CP for all update families, initial conditions, domains, boundary conditions and parameters are coupled on the same probability space using the same clock rings $(P_x)$ and the same unifrom random variables $(\Upsilon_x(t))_{t\in P_x}$.

We next consider a set which will contain the discrepancies between $\cU$-KCM with different initial conditions, based on the the trajectory of the $\{U_0\}$-CP. It can be seen as an analogue of second class particles for the exclusion process \cite{Liggett99}*{Section~III.1}, the envelope probabilistic cellular automaton \cite{Busic13}*{Section~4.2} and is similar to the idea of \cite{Gottschau18}*{Section~1.3}. We define the set $O_t\subset \{x\in\Lambda: \zeta_x(t)=0\}$ of \emph{orange} healthy sites to be the c\`adl\`ag process with jumps at clock ring times $\bigcup_{x\in\Lambda} P_x$ defined as follows. We first set $O_0=\{x\in\Lambda:\xi_x=0\}$, so that all healthy sites are initially orange. Then, for each $x\in\Lambda$ and $t\in P_x$, we set
\begin{equation}
\label{eq:def:Ot}
O_t=\begin{cases}O_{t-}\setminus \{x\}&\zeta_x(t)=1,\\
O_{t-}\cup\{x\}&\zeta_x(t)=0,\exists y\in O_{t-},d(x,y)\le \|\cU\|,\\
O_{t-}&\zeta_x(t)=0,\forall y\in O_{t-},d(x,y)> \|\cU\|.\end{cases}\end{equation}
In words, orange sites appear when a site becomes healthy close to an orange site, but they disappear whenever a site becomes infected.

It is clear that $x\in O_t$ implies $\zeta_x(t)=0$, so that $O_t$ is indeed a subset of the healthy sites in the $\{U_0\}$-CP $\zeta$. {Indeed, by construction, $\bigcup_{x\in\Lambda} P_x$ are the only times when the $\{U_0\}$-CP $\zeta$ may change.}
\begin{lemma}
\label{lem:coupling}
Consider the $\{U_0\}$-CP $\zeta$ on $\Lambda$ with boundary condition $\tau$, initial condition $\xi\in\Omega_\Lambda$ and parameter $q_0$. Also consider the $\cU$-KCM $\eta^\bone$ and $\eta^{\xi'}$ with parameter $q\ge q_0$ on $\Lambda$ with boundary condition $\tau$ and initial conditions $\bone$ and $\xi'\in\Omega_{\Lambda}$ respectively, for some $\xi'\ge \xi$. Then almost surely, we have that 
\begin{equation}
\label{eq:coupling}\left\{x\in\Lambda:\eta_x^{\bone}(t)\neq \eta_x^{\xi'}(t)\right\}\subset O_t\end{equation}
for any $t\ge 0$. In particular, if $O_t=\varnothing$, then $\eta^\bone(t')=\eta^{\xi'}(t')$ for all $t'\ge t$ and $\xi'\ge \xi$. The same holds, if we replace the $\cU$-KCM by the $\cU$-CP.
\end{lemma}
Before proving the lemma, let us prove the following standard fact.
\begin{claim}
\label{claim:comparison}
Consider the $\cU$-KCM $\eta^{\xi'}$ and $\{U_0\}$-CP $\zeta$ as in \cref{lem:coupling}. Then almost surely, for all $t\ge 0$ we have
\begin{equation}\label{eq:comparison}\zeta_x(t)\le \eta^{\xi'}_x(t) \text{ for all } x\in {\Lambda}.\end{equation}
\end{claim}
\begin{proof}
We first consider the case $\Lambda$ finite, so that we can proceed by induction on the clock rings $\bigcup_{x\in\Lambda}P_x$. \Cref{eq:comparison} holds at $t=0$ since $\xi'\ge \xi$. Assume \cref{eq:comparison} holds for all $t'<t\in P_x$ for some $x\in\Lambda$. If $\Upsilon_{x}(t)>q_0$, \cref{eq:comparison} clearly remains true, since $\zeta_x(t)=0$. On the other hand, if $\Upsilon_x(t)\le q_0\le q$ and $\zeta_x(t)=1$, we have two possibilities. If $\zeta_x(t-)=1$, we are done by using \cref{eq:comparison} for $t-$. Instead, if $\zeta_x(t-)=0$, then necessarily $\zeta_{x+U_0}(t-)=\zeta_{x+U_0}(t)=\bone$ (recall \cref{eq:def:CP}). But then \cref{eq:comparison} for $t-$ implies that the constraint $c_x$ (recall \cref{eq:def:cx}) is satisfied in $\eta^{\xi'}(t)$, so that $\eta^{\xi'}_x(t)=1$ (recall \cref{eq:def:KCM}), concluding the proof of the claim for finite $\Lambda$.

Next assume $\Lambda$ is infinite. It follows from the fact that interactions have finite range, that for any $x\in\bbZ^d$ and $t\ge 0$, the states $\zeta_x(t')$ and $\eta_x^{\xi'}(t')$ for all $\xi'\in\Omega_\Lambda$ and $t'\in[0,t]$ coincide with those obtained by the same clock rings and uniform random variables on a finite domain $\Lambda'$ with boundary condition $\bone_{\bbZ^d\setminus\Lambda'}$ with $\Lambda'$ depending on $x$, $t$ and the clock rings. This standard fact can be traced back to \cite{Harris78} (also see e.g.~\cite{Liggett99}*{Section I.1}). Thus, we can apply the result for finite $\Lambda'$ to obtain the one for infinite $\Lambda$.
\end{proof}
\begin{proof}[Proof of \cref{lem:coupling}]
As in the proof of \cref{claim:non-dec}, we may assume that $\Lambda$ is finite and proceed by induction on the clock rings. Since $\xi'\ge \xi$ and initially all healthy sites in $\xi$ are orange, \cref{eq:coupling} holds at $t=0$. Fix $x\in\Lambda$ and $t\in P_x$ and assume that \cref{eq:coupling} holds for any $t'<t$. Further assume for a contradiction that $\eta_x^\bone(t)\neq\eta_x^{\xi'}(t)$, but $x\not\in O_t$. We consider several cases.

\noindent\textbf{Case 1.} Assume $\zeta_x(t)=1$. Then by \cref{claim:comparison} $\eta_x^\bone(t)=\eta_x^{\xi'}(t)=1$, so \cref{eq:coupling} holds, since it holds for $t'<t$ and the only possible change in the left and right hand sides of \cref{eq:coupling} is at $x$. For the $\cU$-CP instead of the $\cU$-KCM, \cref{claim:comparison} is a direct consequence of attractiveness, so the same reasoning applies.

\noindent\textbf{Case 2.} Assume $\zeta_x(t)=0$ and there exists $y\in O_{t-}$ such that $d(x,y)\le \|\cU\|$. Then by definition $x\in O_t$. Moreover, $O_t\setminus\{x\}=O_{t-}\setminus\{x\}$ and 
\[\left\{z\in\Lambda\setminus\{x\}:\eta_z^\bone(t)\neq\eta^{\xi'}_x(t)\right\}=\left\{z\in\Lambda\setminus\{x\}:\eta_z^\bone(t-)\neq\eta^{\xi'}_x(t-)\right\},\]
so \cref{eq:coupling} at $t-$ concludes the proof.

\noindent\textbf{Case 3.} Assume $\zeta_x(t)=0$ and there does not exist $y\in O_{t-}$ such that $d(x,y)\le \|\cU\|$. Then from \cref{eq:coupling} at $t-$ the $\eta^\bone$ and $\eta^{\xi'}$ processes coincide in the neighbourhood of $x$, so they also coincide after the attempted update at $x$, regardless whether it is successful or not.
\end{proof}

\section{Space-time renormalisation of CP to BP with death}
\label{sec:renorm}
Recall from \cref{sec:coupling} that $U_0\in\cU$ is such that $U_0\subset \bbH_u$ for some $u\in S^{d-1}$. In this section we perform a simple renormalisation of the $\{U_0\}$-CP to the $\cU_0$-BP with death {(recall from \cref{subsec:death,subsec:BP} that BP is a CA, so we may consider its version with death)}, where 
\begin{equation}
\label{eq:def:cU0}\cU_0=\left\{\{0,-1\}^{d}\setminus\{0\}\right\}.\end{equation}
The renormalisation is similar to the ones used in \cites{Mareche20,Mareche19a,Mountford19}, where one obtains oriented percolation as a result of the renormalisation. We start by fixing the relevant geometry.

\begin{figure}
\begin{subfigure}{0.34\textwidth}
\centering
\includegraphics{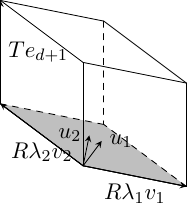}\caption{\label{subfig:box}The box $B$ defined in \cref{eq:def:B} with its base $\hat B${, containing the vectors $u_1,\dots,u_d$,} shaded.}
\end{subfigure}\quad
\begin{subfigure}{0.60\textwidth}
\centering
\includegraphics{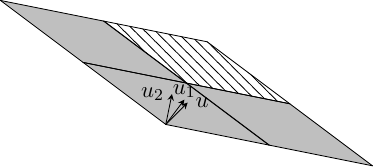}\caption{The $2^d$ box bases $\hat B_{x+y}$ for $y\in\{0,-1\}^d$. The Poisson process points $p_z$ occur in the hatched base $\hat B_x$ in the order of increasing $\<z,u\>$, as indicated by the hatching direction.}
\end{subfigure}
\\
\begin{subfigure}{\textwidth}
\centering
\includegraphics{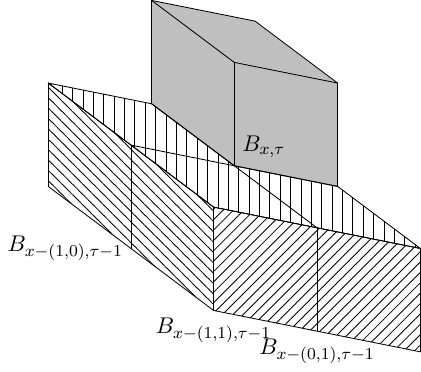}\caption{If the  shaded box $B_{x,\tau}$ is good, we are able to propagate the infection from the three hatched boxes to $B_{x,\tau}$.}
\end{subfigure}
\caption{Illustration of the renormalisation of \cref{sec:renorm} in the case $d=2$.}
\label{fig:renorm}
\end{figure}
Fix linearly independent directions $(u_i)_{i=1}^{d}$  such that: 
\begin{itemize}
\item for all $i\in\{1,\dots,d\}$, $\lambda'_iu_i\in\bbZ^d$ for some $\lambda'_i>0$; \item the $u_i$ are sufficiently close to $u$ so that $U_0\subset\bbH_{u_i}$.
\end{itemize}
The latter condition can be guaranteed thanks to the fact that the $\bbH_{u}$ is defined as the open half-space (recall \cref{subsec:families}). For each $i$ let us decompose, via the Gram--Schmidt algorithm, $u_i=v_i+u_i'$ with $\<v_i,u_j\>=0$ for all $j\neq i$ and $u_i'\in\mathrm{span}(\{u_j:j\neq i\})$, that is, the linear span of the remaining vectors. Note that each $v_i$ can be computed via its own Gram--Schmidt process, rather than the entire family, and these vectors are not necessarily orthogonal (see \cref{subfig:box}). By examining the Gram--Schmidt algorithm, one can check that there also exist $\lambda_i>0$ such that $\lambda_i v_i\in\bbZ^d$. In what follows we consider time as the $d+1$-th coordinate and we abusively identify $u_i\in S^{d-1}$  with $(u_i,0)\in S^{d}$ for all $i$ to lighten notation and similarly we identify $v_i$ with $(v_i,0)$.

Define the \emph{space-time box}  (see \cref{subfig:box})
\begin{align}
\nonumber B&{}=\sum_{i=1}^d(\lambda_iv_i[0,R))+[0,T)e_{d+1}\\
&{}=\left\{a\in \bbR^d:\forall i\in\{1,\dots,d\},\frac{\<a,u_i\>}{\lambda_i\|v_i\|^2 R}\in[0,1)\right\}\times[0,T)\subset \mathbb R^{d+1}\label{eq:def:B}
\end{align} 
for an integer constant $R>0$ chosen sufficiently large depending on $U_0$ and all $u_i,v_i$  and another constant $T>0$ chosen also sufficiently large depending on $R$. We refer to $\hat B=\sum_{i=1}^d(\lambda_iv_i[0,R))\subset \bbR^{d}$ as the \emph{base} of $B$ and to $\hat H=[0,T)$ as its \emph{height}. For any $(x,\tau)=(x_1,\dots,x_d,\tau)\in\bbZ^{d+1}$ we set $B_{x,\tau}=B+\sum_{i=1}^dR\lambda_ix_i{v}_i+T\tau e_{d+1}$, which we view as our renormalised space-time points. We similarly define $\hat B_x$ and $\hat H_\tau$ as the corresponding space and time projections. {This notation is consistent with  that of \cref{sec:initial:condition} for $v'_i=\lambda_i v_i$.}

For each $(x,\tau)\in\bbZ^{d+1}$ we define the event $E_{x,\tau}$ that the point $(x,\tau)$  is \emph{good} if the following two conditions are satisfied:
\begin{itemize}
    \item $\Upsilon_z(t)\le q_0$ for all $z\in\bigcup_{y\in\{0,-1\}^d}\hat B_{x+y}$ and $t\in (\hat H_\tau\cup \hat H_{\tau-1})\cap P_z$, where we recall from \cref{subsec:KCM} that $P_z$ is the Poisson process associated to vertex $z$;
    \item For each $z\in \hat B_x$ there exists $p_z\in P_z\cap \hat H_{\tau-1}$ and $\<z,u\>>\<z',u\>$ implies $p_z>p_{z'}$.
\end{itemize}
The box is called \emph{bad} if it is not good. In words, a good space-time box has a suitable sequence of clock rings and all updates in the box and its $\cU_0$-neighbours attempt to infect the corresponding sites. Notice that $E_{x,\tau}$ only depends on the Poisson process points $P_z\cap\{t\}$ and the uniform random variables $\Upsilon_{z}(t)$ with $(z,t)\in\bigcup_{y\in\{0,-1\}^d} B_{x+y,\tau}\cup  B_{x+y,\tau-1}$,

The following lemma provides the desired coupling between the $\{U_0\}$-CP with parameter $q_0$, denoted $\zeta$, and a $\mathcal U_0$-BP with death, denoted $\widetilde\omega$. In fact, the $\widetilde\omega$ process will not quite be a $\cU_0$-BP with death, but is defined in the same way, given the Bernoulli variables $\xi_{x,t}$ from \cref{eq:def:death}, which will however not be i.i.d.\ (the tilde is there to remind us of this difference). Nevertheless, we will somewhat abusively refer to it as a $\cU_0$-BP with death despite this.
\begin{lemma}
\label{lem:CP-BP-comaprison}
Consider the $\cU_0$-BP $\widetilde\omega$ with death on $\bbZ^d$ with death marks {$\xi_{x,\tau}=1$ }at points $(x,\tau)\in\bbZ^{d}\times\bbN$ such that $B_{x,\tau}$ is bad. Further let its initial condition be given by the configuration where infected sites $x$ are exactly  the sites $x\in\bbZ^d$ such that:
\begin{itemize}
    \item $\hat B_x$ is fully in state 1 in the initial condition $\zeta{(0)}$ of the $\{U_0\}$-CP with parameter $q_0$;
    \item  $\Upsilon_z(t)\le q_0$ for all $z\in \hat B_x$ and $t\in\hat H_0\cap P_z$. 
\end{itemize} 
Then for all $(x,\tau)\in\bbZ^{d}\times\bbN$ and $(y,t)\in B_{x,\tau}$, if $\zeta_y(t)=0$, then $\widetilde\omega_x(\tau)=0$.
\end{lemma}
\begin{proof}
We prove the statement by induction on $\tau\ge 0$. The base case follows from the definition of the initial condition of $\widetilde\omega$. Therefore, let $\tau\ge 1$ and assume for a contradiction that $\zeta_y(t)=0$, $\widetilde\omega_x(\tau)=1$ for some $(y,t)\in B_{x,\tau}$. If $\widetilde\omega_x(\tau-1)=1$, then by induction $\zeta_y((\tau-1)T)=1$ and by the first condition for $B_{x,\tau}$ being good (which holds, because otherwise $\widetilde\omega_x(\tau)=0$, since bad boxes are deaths) this remains true until time $t$, leading to a contradiction. Thus, we may assume that $\widetilde\omega_x(\tau-1)=0$, so that (by the definition \cref{eq:def:BP} of BP) $\widetilde\omega_{x+z}(\tau-1)=1$ for all $z\in\{0,-1\}^d\setminus\{0\}$ and $\zeta_{w}((\tau-1)T)=1$ for all $w\in\hat B_{x+z}$ by induction hypothesis. Since the box $B_{x,\tau}$ is good, there has been no attempt to put $0$ at $y$ after time $\tau T$ in the contact process, so it suffices to show that $\zeta_y(\tau T)=1$ to reach a contradiction.

We claim that by our choice of geometry (see \cref{fig:renorm}), 
\begin{equation}
\label{choice-of-geometry}
U_0+\hat B_{x}\subset \bigcup_{z\in\{0,-1\}^d}\hat B_{x+z}.    
\end{equation}
 To see this, notice that 
\begin{align*}
\hat B_x&{}=\left\{a\in\bbR^d:\forall i\in\{1,\dots,d\},\frac{\<a,u_i\>}{\lambda_i\|v_i\|^2 R}\in[ x_i,x_i+1)\right\},\\
\bigcup_{z\in\{0,-1\}^d} \hat B_{x+z}&{}=\left\{a\in\bbR^d:\forall i\in\{1,\dots,d\},\frac{\<a,u_i\>}{\lambda_i\|v_i\|^2 R}\in[ x_i-1,x_i+1)\right\}.\end{align*}
Therefore, for any $a\in \hat B_x$, $b\in a+U_0$ and $i\in\{1,\dots,d\}$ we have \[\lambda_i\|v_i\|^2 R(x_i-1)\le \<a,u_i\>-\|\cU\|\le \<b,u_i\><\<a,u_i\><\lambda_i\|v_i\|^2 R(x_i+1),\]
if we choose $R$ large enough so that $\lambda_i\|v_i\|^2 R>\|\cU\|$ for every $i$. Hence, $b\in\bigcup_{z\in\{0,-1\}^d}\hat B_{x+z}$ as claimed.

Finally, observe that by the second condition for $B_{x,\tau}$ being good, there has been a sequence of attempts at times $p_a\in\hat H_{\tau-1}$ to put $1$ at each site $a\in\hat B_{x}$. Since the sequence is in the order of increasing scalar product with $u$ and $a+U_0$ is contained in $(a+\bbH_u)\cap \bigcup_{z\in\{0,-1\}^d}\hat B_{x+z}$ (again, provided that $R$ is large enough), the constraint $c_a({\zeta(p_a)})$ is fulfilled for each of them, so we are done.
\end{proof}
The following corollary will be more convenient for our purposes.
\begin{corollary}
\label{cor:CP-BP-comparison}
Fix $U_0$ as above and $\delta>0$. There exist $\varepsilon_0>0$ and $R_1>0$ such that for any $R\ge R_1$ there exists $T_1$ such that for any $T\ge T_1$ there exists $
\varepsilon_2>0$ such that for any $q_0\in[1-\varepsilon_2,1]$ the following holds. Consider the $\{U_0\}$-CP $\zeta$ with parameter $q_0$ and initial condition {given by $\zeta_y(0)=\xi_x$ for all $y\in\hat B_x$ and $x\in \bbZ^d$, where $\xi\sim\mu_{1-\varepsilon_0}$.} Then the trajectory at times $\tau\ge 1$ of the $\cU_0$-BP with $\delta$ death and initial condition $\bone$ is stochastically dominated by the process given by \[\hat \omega_x(\tau)=\1_{\forall (y,t)\in B_{x,\tau-1}, \zeta_{y}(t)=1}.\]
\end{corollary}

\begin{proof}
Fix $\varepsilon_0>0$ small enough depending on $\delta$. Taking $T$ large after $R$ and then $\varepsilon_2>0$ small, it is clear that the probability of a box being good can be made larger than $1-\varepsilon_0$. Moreover, good boxes together with the initial condition $\widetilde\omega(0)$ from \cref{lem:CP-BP-comaprison} form a percolation with bounded range of dependence. Therefore, by the Liggett--Schonmann--Stacey Theorem \cite{Liggett97} it stochastically dominates an independent Bernoulli field with parameter $1-\delta$, provided $\varepsilon_0$ is small enough. By attractiveness of $\cU_0$-BP, this together with \cref{lem:CP-BP-comaprison} completes the proof. Indeed, we simply observed that the first step of $\cU_0$-BP with $\delta$ death and initial condition $\bone$ has distribution $\mu_{1-\delta}$. 
\end{proof}

\section{Exponential decay}
\label{sec:Toom}
In this section we establish several exponential decay properties. It can be viewed independently of the rest of the paper and will entail results of independent interest. We therefore adopt a rather abstract and general framework, to keep the approach as flexible as possible. {It will be convenient to work in $\bbZ^D$ with $D\ge 1$, which will play the role of $d+1$ in our original setting.}
\subsection{Decorated set systems}
\begin{definition}[$k$-connectivity]
\label{def:kconnected}
Fix a positive real $k$. We say that a set $X\subset {\bbR}^{D}$ is \emph{$k$-connected}, if for every $x,y\in X$ there exists a sequence $x_0=x,x_1,\dots, x_m=y$ of distinct elements of $X$ such that $d(x_i,x_{i+1})\le k$ for all $i\le m-1$. We call such a sequence a \emph{$k$-connected path} with \emph{endpoints} $x$ and $y$.
\end{definition}

\begin{definition}[Decorated set system]
\label{def:decorated:set:systems}
For any set $Z\subset \bbZ^d$ we fix an arbitrary set $\Gamma_Z$ of \emph{possible decorations}. We allow some $\Gamma_Z$ to be empty, making the corresponding sets $Z$ impossible to decorate. A \emph{decorated set} is a pair $(Z,\gamma)$ with $Z\subset\bbZ^D$ nonempty and bounded and $\gamma\in\Gamma_Z$. Two decorated sets $(Z_1,\gamma_1)$ and $(Z_2,\gamma_2)$ are called \emph{disjoint}, if $Z_1\cap Z_2=\varnothing$. A \emph{decorated set system} is a probability measure $\bbP$ and a function $E$ that associates to each decorated set $(Z,\gamma)$ an event $E(Z,\gamma)$ in such a way that for any finite set of disjoint decorated sets $(Z_i,\gamma_i)_{i\in I}$ we have 
\begin{equation}
\label{eq:negative:association}
\bbP\left(\bigcap_{i\in I}E(Z_i,\gamma_i)\right)\le \prod_{i\in I}\bbP\left(E(Z_i,\gamma_i)\right).
\end{equation}
For $x\in\bbZ^D$, we denote by 
\[E(x)=\bigcup_{(Z,\gamma):x\in Z\subset\bbZ^D,\gamma\in\Gamma_Z}E(Z,\gamma)\]
the event that there exists a decorated set containing $x$ whose event occurs.
\end{definition}
In all applications of this construction below, we will actually have equality in \cref{eq:negative:association} but we work under this more general condition, as the proof of \cref{prop:set:systems} works exactly the same with inequality or with equality. In what follows for any $X\subset \bbR^{D}$ we denote $\diam(X)=\sup_{x,y\in X}(d(x,y))$ with the convention $\diam(\varnothing)=-\infty$.
\begin{proposition}
\label{prop:set:systems}
Consider a decorated set system. Assume that for some $C>1$ and $\epsilon>0$ small enough depending on $C$ the following hold:
\begin{enumerate}
    \item\label{cond:1} For all decorated sets $(Z,\gamma)$ we have $\diam(Z)\le C|Z|$.
    \item\label{cond:2} For every $z\in\bbZ^{D}$ the number of decorated sets $(Z,\gamma)$ such that $|Z|=m$ and $z\in Z$ is at most $C^m$.
    \item\label{cond:3} For all decorated sets $(Z,\gamma)$ we have $\bbP(E(Z,\gamma))\le \epsilon^{|Z|/C}$.
\end{enumerate}
Fix $n,k\ge 1$ and $x\in\bbZ^{D}$. Let $\cE(x,n,k)$ denote the event that there exist $y\in\bbZ^D$ with $d(x,y)\ge n$ and a $k$-connected path $P$ with endpoints $x$ and $y$ such that $E(p)$ occurs for each $p\in P$. Then
\[\bbP(\cE(x,n,k))\le \epsilon^{n/({7}C+{7}k)^2}.\]
\end{proposition}
Note that the second condition implies that $\Gamma_Z$ is finite for every finite $Z$. {\Cref{prop:set:systems} will be proved in \cref{subsec:proof:set:systems}, but before that, let us provide a few applications to make the abstract setting more concrete.}

\subsection{Applications}
While it is not hard to imagine examples of decorated set systems satisfying the conditions of \cref{prop:set:systems}, let us give a more explicit toy example to get used to the notion before turning to more interesting applications based on Toom contours and variants thereof.
\begin{example}
Consider a field of i.i.d.\ Bernoulli random variables $\xi_x$ for $x\in \bbZ^{D}$. For each finite non-empty $Z\subset \bbZ^{D}$ the set of decorations $\Gamma_Z$ is empty if $Z$ is not a 1-connected path and a singleton otherwise. The event $E(Z,\gamma)$ corresponds to $\bigcap_{x\in Z}\{\xi_x=1\}$. Then \cref{eq:negative:association} is satisfied by independence. Condition \ref{cond:1} of \cref{prop:set:systems} follows by 1-connectedness. Condition \ref{cond:2} holds, because one can encode a 1-connected path by the sequence of its increments, so the number of paths containing $z\in\bbZ^{D}$ of cardinality $n$ is at most $n(2D)^n$. Condition \ref{cond:3} is also verified, since $\bbP(E(Z,\gamma))=(\bbP(\xi_0=1))^{|Z|}$. Hence, if the Bernoulli variables have a sufficiently small parameter, \cref{prop:set:systems} yields an exponentially small bound on the probability that one can find a $k$-connected path starting at 0 such that each of its points belongs to a $1$-connected path with all Bernoulli variables equal to 1. Thus, in this case this degenerates into looking for a $k$-connected path in the set of sites with Bernoulli variable equal to 1, so the conclusion of the proposition is a classical fact.
\end{example}

The next corollary will be used to control the $\cU_0$-BP with death we recovered in \cref{cor:CP-BP-comparison}. We formulate it more generally for cellular automata with death.
\begin{corollary}
\label{cor:Toom}
Let $k$ be a positive real number and let the map $\phi$ define a CA which is attractive and an eroder, as defined in \cref{subsec:death}. Then there exist $c>0$ and $\delta_0>0$ such that for all $\delta\in(0,\delta_0]$ the {$\phi$-CA} $\omega^{\bar\nu}$ with  death parameter $\delta$ and with initial condition given by its upper invariant measure $\bar\nu$ satisfies the following. Denote by $A_{x,t}$ the maximal $k$-connected component containing $(x,t)\in\bbZ^{d+1}$ with $\omega^{\bar\nu}_a(s)=0$ for all $(a,s)\in A_{x,t}$. It holds that
\[\bbP\left(\diam(A_{x,t})\ge \ell\right)\le\delta^{c(\ell+1)}.\]
\end{corollary}
\begin{remark}
Let us note that \cref{cor:Toom} provides a much more straightforward and general proof of several of the main results of \cite{Poncelet13} (see Theorems 5, 6 and 7 there).
\end{remark}
\begin{proof}[Proof of \cref{cor:Toom}]
\Cref{cor:Toom} follows directly from \cref{prop:set:systems}, applied to the decorated set system given by Toom contours \cites{Toom80,Swart22} and their presence. Since these notions are rather technical to define, while the details of the definition are irrelevant, let us instead highlight the high level viewpoint, referring to \cite{Swart22} for more details (the facts we will need were actually already known since \cite{Toom80}). 

A Toom contour rooted at $z\in\bbZ^{d+1}$ consists of a finite set $Z\subset \bbZ^{d+1}$ (with $z\in Z$) equipped with a complicated decoration taking the form of a connected coloured oriented multigraph with vertex set $Z$, such that the endpoints of each edge are at most at some bounded mutual distance (depending on the support of $\phi$). Connectedness readily implies condition \ref{cond:1}.

In each Toom contour one can identify a set $Z_*\subset Z$ of sinks. A contour is said to be present if $\xi_{x,t}=1$ (recall \cref{subsec:death}) for each sink $(x,t)$, so that disjoint decorated sets occur independently. In particular, Toom contours form a decorated set system and their probability of occurrence is $\delta^{|Z_*|}$. A key and highly non-trivial fact \cite{Swart22}*{Theorem 7} is that if $\omega_x^{\bar\nu}(t)=0$, then some non-empty finite Toom contour rooted at $(x,t)$ occurs. Moreover, \cite{Swart22}*{Lemma 13} ensures that the number of edges (and, therefore, the number of vertices by connectedness) of a Toom contour is at most a constant multiple of $|Z_*|$, thus proving condition \ref{cond:3} . Finally, \cite{Swart22}*{Lemma 14} shows that the number of Toom contours containing a given point and  with $N$ edges is at most exponential in $N$ and, therefore, in the number of vertices by the previous result, so condition \ref{cond:2} is also satisfied.
\end{proof}
The next corollary will not be used in the proof of our main results, but we include it, since it is of independent interest and follows analogously.
\begin{corollary}
\label{cor:BP}
Consider $\cU'$-BP $\omega$ with subcritical $\cU'$. Fix $p>0$ and let $C_0$ be the $k$-connected component containing the origin in $\{x\in\bbZ^d:\lim_{t\to\infty}\omega_x^{\mu_p}(t)=1\}$. Then there exists $c=c(\cU',k)>0$ such that for all $p>0$ small enough and $\ell\ge 0$ we have
\[\bbP(\diam(C_0)\ge \ell)\le p^{c(\ell+1)}.\]
\end{corollary}
We recall that non-subcritical models are exactly those with $\qc[BP]=0$ \cites{Balister24,BalisterNaNb}, so $C_0=\bbZ^d$ almost surely for all $p>0$ and it is meaningless to consider them in the above sense.
\begin{remark}
\Cref{cor:BP} provides a positive answer to \cite{Balister16}*{Question 12} in the perturbative regime. As pointed out in \cite{Hartarsky21}*{Section 7.1.2}, the stated exponential decay cannot hold for all $p<\qc[BP]$ in general. Nevertheless, in the spirit of \cite{AlvesNaN} one could expect that it does hold up to a different critical threshold $p_{\mathrm{c}}(k)\le \qc[BP]$, past which the diameter is a.s.\ infinite. On a different note, \cref{cor:BP} was proved by more classical means in \cite{Blanquicett20}*{Theorem 4.2} for \emph{trivial} subcritical update families in two dimensions for $k=1$. Thus, \Cref{cor:BP} vastly generalises this result and solves \cite{Blanquicett20}*{Problem~6.1}.
\end{remark}

\begin{proof}[Proof of \cref{cor:BP}]\Cref{cor:BP} follows from {\cref{prop:set:systems}} essentially along the same lines as \cref{cor:Toom} follows from \cref{prop:set:systems}, but using a different decorated set system. Namely, we consider the space embeddings of shattered contours, which are the central object of study in \cite{Hartarsky22Toom}\footnote{Note that in \cite{Hartarsky22Toom} the roles of 0 and 1 are exchanged with respect to the present work.}, similarly to Toom contours in \cite{Swart22}. Again, the definition of these objects, which are projections of equivalence classes of Toom contours, is rather technical and unimportant for us, so we refer to \cite{Hartarsky22Toom} for those details. Instead, let us indicate that the presence of a finite non-empty shattered contour rooted at $x\in\bbZ^d$ is implied by $\lim_{t\to\infty}\omega_x^{\mu_p}(t)=1$ \cite{Hartarsky22Toom}*{Corollary 4.3}; their numbers of edges and vertices are bounded by a constant multiple of the number of sinks \cite{Hartarsky22Toom}*{Lemma 5.2}; the number of shattered Toom contours rooted at a given point is at most exponentially large in the number of sinks \cite{Hartarsky22Toom}*{Lemma 5.3}.
\end{proof}

Thus, our only remaining task in this section is to prove \cref{prop:set:systems}. Before doing so, let us mention a natural question closely related to \cref{cor:Toom}.
\begin{question}
Fix $\varepsilon>0$. Is it true that for the $\cU_0$-BP with $\delta$ death, with $\delta>0$ small enough, the upper invariant measure stochastically dominates $\mu_{1-\varepsilon}$? More generally, is this true for any attractive eroder?
\end{question}

\subsection{Proof of Proposition~\ref{prop:set:systems}}
\label{subsec:proof:set:systems}
We fix a decorated set system and $C$, $\epsilon$, $n$, $k$, and $x$ as in \cref{prop:set:systems}. Since our sets can be quite fuzzy we begin by regularising them.
\begin{definition}
\label{def:bar}
Given a finite non-empty $Z\subset\bbZ^d$, we set 
\[\bar Z:=\left\{x\in\bbZ^d:d(x,Z)\le 3(1+\diam(Z))\right\}.\]
\end{definition}
Let us fix a $k$-connected path $P=(p_0=x,p_1,\dots,p_l=y)$ with $d(x,y)\ge n$. We further fix decorated sets $(Z_p,\gamma_p)$ with $p\in Z_p$ for each $p\in P$. We next run the following algorithm.

\begin{algorithm}
\label{algo}
Define $i_0=0$, $I_0=\{0\}$ and $X_0=\bar Z_{p_0}$ and initialise $t=0$. While $P\not\subset X_t$, repeat the following, then return $(I_t,X_t,t)$. Increment $t$ by setting $t:=t+1$. Set $i_{t}=\min\{j\le l:p_j\not\in X_{t-1}\}$. Let $J_{t}=\{j\in I_{t-1}: Z_{p_j}\cap Z_{p_{i_t}}\neq \varnothing\}$. Set $I_{t}=\{i_t\}\cup (I_{t-1}\setminus J_t)$. Set $X_t=\bigcup_{j\in I_t}\bar Z_{p_j}$.
\end{algorithm}
By definition, if the algorithm terminates and outputs $(I_t,X_t,t)$, then $P\subset X_t$. Moreover, by induction we have that for all $t'\le t$ and $a,b\in I_{t'}$ with $a\neq b$, it holds that 
\begin{equation}
\label{eq:Z:disjointness}
Z_{p_a}\cap Z_{p_b}=\varnothing,\end{equation}
using the definition of $J_{t'}$. To see that \cref{algo} terminates, it suffices to see that $X_{t'}\cap P$ is strictly increasing in $t'$, since $P$ is finite (on the other hand, we note that $I_{t'}$ is not necessarily monotone in $t'$). In order to prove this, we first show that $X_{t'}\cap P$ is non-decreasing and then exhibit an element which is in $X_{t'}\cap P$, but not in $X_{t'-1}\cap P$. 

\begin{claim}
\label{claim:non-dec}
For any $t'\in\{1,\dots,t\}$ we have $X_{t'}\cap P\supset X_{t'-1}\cap P$.
\end{claim}
\begin{proof}
By the definitions of $X_{t'}$ and $I_{t'}$, we have that 
\begin{equation}
\label{eq:barZ}
X_{t'-1}\setminus X_{t'}\subset\bigcup_{j\in J_{t'}}\bar Z_{p_j}\setminus\bar Z_{p_{i_{t'}}}.
\end{equation}
Thus, it remains to show that $\bar Z_{p_{i_{t'}}}\supset\bar Z_{p_j}$ for all $j\in J_{t'}$.

Fix $j\in J_{t'}\subset I_{t'-1}$, so $Z_{p_j}\cap Z_{p_{i_{t'}}}\neq \varnothing$ by the definition of $J_{t'}$. Then
\[\diam\left(Z_{p_{i_{t'}}}\right)\ge d\left(p_{i_{t'}},p_j\right)-\diam\left(Z_{p_j}\right)\ge d\left(p_{i_{t'}},p_j\right)/3+2+\diam\left(Z_{p_j}\right),\]
where we first used the triangle inequality, then the fact that $p_{i_{t'}}\not\in \bar Z_{{p}_j}$ by definition of $i_{t'}$ and $X_{t'-1}$. But then $\bar Z_{p_{i_{t'}}}\supset \bar Z_{p_j}$ by the triangle inequality and we are done.
\end{proof}
We next claim that for any $t'\in\{1,\dots,t\}$ we have 
\begin{equation}
\label{claim:strict}
p_{i_{t'}}\in X_{t'}\setminus X_{t'-1}=\bar Z_{p_{i_{t'}}}\setminus \bigcup_{j\in I_{t'-1}}\bar Z_{p_j}.\end{equation}
Indeed, if $p_{i_{t'}}\in \bar Z_{p_j}$ for some $j\in  I_{t'-1}$, that would imply $p_{i_{t'}}\in X_{t'-1}$ by the definition of $X_{t'-1}$, but this contradicts the definition of $i_{t'}$. 

Combining \cref{claim:non-dec,claim:strict}, we get that \cref{algo} does terminate, so $I_t,X_t,t$ are well defined.
\begin{claim}
\label{claim:connected}
$X_t$ is $k$-connected.
\end{claim}
\begin{proof}
We prove by induction that $X_{t'}$ is $k$-connected for all $t'\in \{0,\dots,t\}$. The base follows since $X_0=\bar Z_{p_0}$ is $k$-connected, by \cref{def:bar,def:kconnected} and $k\ge 1$. By \cref{claim:strict}, $X_{t'}\setminus X_{t'-1}\subset \bar Z_{p_{i_{t'}}}\subset X_{t'}$, the last inclusion using the definition of $X_{t'}$ and $I_{t'}$. Moreover, by the proof of \cref{claim:non-dec} $X_{t'}\supset X_{t'-1}$. Thus, since $\bar Z_{p_{i_{t'}}}$ is $k$-connected, it remains to show that $d(X_{t'-1},\bar Z_{p_{i_{t'}}})\le k$. In fact, the stronger statement $d(p_{i_{t'}},X_{t'-1})\le k$ holds, because the definition of $i_{t'}$ gives $p_{i_{t'}-1}\in X_{t'-1}$, and $i_{t'}\neq 0$ for $t'>0$, since $X_{t'-1}\supset X_0=\bar Z_{p_0}\ni p_0$ and because consecutive $p_j$ are at distance at most $k$.
\end{proof}
\begin{definition}[Chain]
\label{def:chains}
A \emph{chain starting at $x$ of length at least $n$} is a sequence of disjoint decorated sets $(V_j,\gamma_j)_{j=1}^m$ such that $d(\bar V_j,\bar V_{j+1})\le k$ for all $j\le m-1$, $x\in \bar V_1$ and there exists $y\in \bar V_m$, such that $d(x,y)\ge n$. 
\end{definition}

By \cref{claim:connected,eq:Z:disjointness} we can extract from $I_t$ a sequence $i'_1,\dots,i'_m$ such that $(V_j,\gamma_j)_{j=1}^m$ is a chain starting at $x$ of length at least $n$, where $V_j=Z_{p_{i'_j}}$ and $\gamma_j=\gamma_{p_{i'_j}}$ for $j\in\{1,\dots,m\}$. Thus, recalling $\cE(x,n,k)$ from \cref{prop:set:systems}, we have proved the following.
\begin{lemma}
\label{lem:chains}
If $\cE(x,n,k)$ occurs, then there exists a chain $(V_j,\gamma_j)_{j=1}^m$ starting at $x$ of length at least $n$ such that $\bigcap_{j=1}^mE(V_j,\gamma_j)$ occurs.
\end{lemma}

We are now ready to conclude the proof of \cref{prop:set:systems} by a union bound over all such chains, since their decorated sets are disjoint, so the events involved are negatively correlated. 
\begin{proof}[Proof of \cref{prop:set:systems}]
By \cref{lem:chains,def:chains,def:decorated:set:systems},
\begin{equation}
\label{eq:proba:bound}\bbP\left(\cE(x,n,k)\right)\le\sum_{(V_j,\gamma_j)_{j=1}^m}\prod_{j=1}^m\epsilon^{|V_j|/C},\end{equation}
where the sum is over all chains starting at $x$ of length at least $n$. Observe that by condition \ref{cond:1} of \cref{prop:set:systems}, we have 
\begin{equation}
\label{eq:n:bound}n\le \sum_{i=1}^m\left(k+\diam\left(\bar V_i\right)\right)\le mk+{6}m+{7}\sum_{i=1}^m\diam\left(V_i\right)\le {7(C+k)}\sum_{i=1}^m |V_i|.\end{equation}

Further note that for any $i\in\{1,\dots,m-1\}$ the distance between an arbitrarily chosen point in $V_i$ and one in $V_{i+1}$ is at most
\[k+\diam\left(\bar V_i\right)+\diam\left(\bar V_{i+1}\right)\le k+{12}+{14}C\max(|V_i|,|V_{i+1}|).\]
Therefore, the number of ways to fix the positions of one distinguished point $\pi_i$ in each $V_i$ with $\pi_1=x$ is at most
\begin{equation}
\label{eq:distinguished:bound}\prod_{i=1}^{m-1} (1+2(k+{12}+{14}C\max(|V_i|,|V_{i+1}|)))^d\le \prod_{i=1}^m30(k+C|V_i|)^{2d}.\end{equation}

Combining \cref{eq:proba:bound,eq:n:bound,eq:distinguished:bound} with condition \ref{cond:2} of \cref{prop:set:systems}, the probability we seek to bound in \cref{prop:set:systems} is at most
\begin{multline}
\label{eq:sum:mn}\sum_{\substack{m,n_1,\dots,n_m\ge 1\\\sum_{i=1}^m n_i\ge n/(7(C+k))}}\prod_{i=1}^mC^{n_i}30(k+Cn_i)^{2d}\epsilon^{n_i/C}\\
\le 30\sum_{N=\lceil n/(7C+7k)\rceil}^\infty2^NC^{N}(C+k)^{2dN}\epsilon^{N/C}\le 30\epsilon^{n/(2C(7C+7k))}.\end{multline}
For $\epsilon$ small enough \cref{eq:sum:mn} is clearly at most $\epsilon^{n/{({7}C+{7}k)^2}}$ as desired, completing the proof of \cref{prop:set:systems}.
\end{proof}

\section{Assembling Theorem~\ref{th:convergence}}
\label{sec:convergence}
In this section we assemble the results of \cref{sec:initial:condition,sec:coupling,sec:renorm,sec:Toom} in order to prove \cref{th:convergence}. Therefore, let us fix an update family $\cU$ which is not trivial subcritical (otherwise $\qct[KCM]=1$ and there is nothing to prove) and $U_0\in\cU$ such that $U_0\subset\bbH_u$ for some $u\in S^{d-1}$. Let $v_i'=v_i\lambda_i$ for all $i\in\{1,\dots,d\}$, where $v_i,\lambda_i$ are chosen as in \cref{sec:renorm}. Let $\delta_0$ be as in \cref{cor:Toom} for $k=\sqrt{d+1}$ and $\phi$ be the map corresponding to $\cU_0$-BP (recall \cref{eq:def:cU0}), which is clearly an eroder attractive cellular automaton. Then let $\varepsilon_0$ be as in \cref{cor:CP-BP-comparison}, setting $\delta=\delta_0$. Fix $\alpha\in(0,1-\qct[KCM])$. Let $R=\max(R_0,R_1)$, where $R_0$ is as in \cref{lem:initial:cond} and $R_1$ is as in \cref{cor:CP-BP-comparison}. Let $T=\max(T_0,T_1)$, where $T_0$ is as in \cref{lem:initial:cond} and $T_1$ is as in \cref{cor:CP-BP-comparison}. Let $\varepsilon_1$ be as in \cref{lem:initial:cond}. Finally, let $p\in [\qct[KCM]+\alpha,1]$ and $\varepsilon=\min(\varepsilon_1,\varepsilon_2,1-\qct[KCM]-\alpha)$, with $\varepsilon_2$ from \cref{cor:CP-BP-comparison}, and $q\in[1-\varepsilon,1]$.

Let $\eta^{\mu_p}$ denote the $\cU$-KCM on $\bbZ^d$ with parameter $q$ and initial condition with law $\mu_p$ and $\eta^{\mu_q}$ be the stationary $\cU$-KCM with the same parameter. The initial conditions are coupled so that $\eta_x^{\mu_q}(0)\le \eta_x^{\mu_p}(0)$ for all $x\in\bbZ^d$, if $q\le p$ and $\eta_x^{\mu_q}(0)\ge \eta_x^{\mu_p}(0)$ for all $x\in\bbZ^d$, if $p\le q$. Note, however, that, due to the non-attractiveness of KCM, this inequality need not be preserved by the dynamics. 

Since $\eta^{\mu_q}$ is stationary, we get that 
\begin{equation}
\label{eq:convergence:coupling}
\left|\bbE\left[f(\eta^{\mu_p}(t))-\mu_q(f)\right]\right|\le 2\|f\|_\infty\cdot\bbP\left(\eta_{S}^{\mu_p}(t)\neq \eta^{\mu_q}_{S}(t)\right),\end{equation}
where $S$ denotes the support of the local function $f$. By a union bound over the sites of $S$ and translation invariance, for all $t\ge 0$
\begin{equation}
\label{eq:convergence:single:site}\bbP\left(\eta_{S}^{\mu_p}(t)\neq \eta^{\mu_q}_{S}(t)\right)\le |S|\cdot\bbP\left(\eta_0^{\mu_p}(t)\neq \eta_0^{\mu_q}(t)\right).\end{equation}

We start by running the two KCM up to time $T$. Then \cref{lem:initial:cond} gives that $\min(\eta^{\mu_{p}}_y(T),\eta^{\mu_q}_y(T))\ge \xi_x$ for all $x\in\bbZ^d$, $y\in\hat B_x$, where $\xi\sim\mu_{1-\varepsilon_0}$ is suitably coupled with the two KCM.

Let $\zeta$ be the $\{U_0\}$-CP on $\bbZ^d$ starting at time $T$ with  parameter $q_0={1-}\varepsilon_2\le q$ and initial condition given by $\zeta_y(T)=\xi_x$ for all $y\in \hat B_x$ and $x\in\bbZ^d$. Further recall its orange healthy sites $O_t$ from \cref{sec:coupling}, which we now define w.r.t.\ the initial time $T$ instead of $0$. By the Markov property and \cref{lem:coupling} applied once to $\xi'=\eta^{\mu_p}(T)$ and once to $\xi'=\eta^{\mu_q}(T)$ gives that for all $t\ge T$
\begin{equation}
\label{eq:eta:Ot}\bbP\left(\eta_0^{\mu_p}(t)\neq \eta_0^{\mu_q}(t)\right)\le \bbP(0{\in} O_t).\end{equation}

Recall \cref{eq:def:norm:U,eq:def:Ot,def:kconnected}. Observe that by construction $0\in O_t$ implies not only that $\zeta_0(t)=0$, but also that there exists a $\|\cU\|$-connected set $K\subset \bbZ^d\times[T,\infty)$ in space-time which contains $(0,t)$, intersects the hyperplane $\bbZ^d\times\{T\}$ and satisfies $\zeta_x(\theta)=0$ for all $(x,\theta)\in K$. Indeed, whenever a site is added to $O_t$, it has to be at distance at most $\|\cU\|$ from an orange site. Let us call $E_t$ the event that $(0,t)$ belongs to a $\|\cU\|$-connected component of space-time points $(x,\theta)\in\bbZ^d\times[0,t]$ with $\zeta_x(\theta)=0$, that intersects $\bbZ^d\times\{T\}$. Then we just showed that
\begin{equation}
\label{eq:Ot}\bbP(0{\in }O_t)\le \bbP(E_t).
\end{equation}

We are then ready to apply \cref{cor:CP-BP-comparison}. To that end, let $\omega$ be the $\cU_0$-BP with $\delta_0$-death and initial condition $\bone$ (recall \cref{eq:def:cU0}). \Cref{cor:CP-BP-comparison} gives that we can couple $\omega$ and $\zeta$ in such a way that for all $x\in\bbZ^d$ and $\tau\ge 1$, $\omega_x(\tau)=1$ implies that $\zeta_y(t)=1$ for all $(y,t)\in B_{x,\tau}$. 

Then for $t\ge T$ the event $E_t$ implies that there is a $\sqrt{d+1}$-connected  component of space-time points $(x,\tau)\in\bbZ^{d+1}$ such that $\omega_x(\tau)=0$ containing both $(0,\lfloor t/T\rfloor)$ and a point in $\bbZ^d\times\{1\}$. Let us denote this event by $F_{\lfloor t/T\rfloor-1}$, so that that for any $\tau\ge 0$
\begin{equation}\label{eq:Et}
\bbP(E_t)\le \bbP\left(F_{\lfloor t/T\rfloor-1}\right).\end{equation}

Since the $\cU_0$-BP with $\delta_0$ death is attractive and $\bone$ is the maximal initial condition, denoting by $F_\tau^{\bar\nu}$ the event $F_\tau$ with $\omega$ replaced by the process $\omega^{\bar\nu}$ with initial condition distributed according to its upper invariant measure, we get that for some $c>0$ and any $\tau\ge 0$
\begin{equation}\label{eq:Ftau}
\bbP(F_\tau)\le \bbP\left(F^{\bar\nu}_\tau\right)\le \delta_0^{c(\tau+1)},\end{equation}
where we applied \cref{cor:Toom}. Note that $c$ is now an absolute constant, as it no longer depends on $\cU$, but only on $\cU_0$, which is fixed by \cref{eq:def:cU0}.

Putting \cref{eq:Ftau,eq:Et,eq:Ot} together, we get that for all $t\ge T$
\begin{equation}
    \label{eq:decay:Ot}
    \bbP(0\in O_t)\le \delta_0^{c\lfloor t/T\rfloor}.
\end{equation}
Further recalling \cref{eq:eta:Ot,eq:convergence:coupling,eq:convergence:single:site}, this yields
\[\left|\bbE\left[f(\eta^{\mu_p}(t))-\mu_q(f)\right]\right|\le 2\|f\|_\infty|S|\delta_0^{c\lfloor t/T\rfloor},\]
completing the proof of \cref{th:convergence}.

\section{Renormalisation of CP to LPP}
\label{sec:LPP}
In this section we perform a renormalisation somewhat similar to the one from \cref{sec:renorm}. Recall $\cU_0$ from \cref{eq:def:cU0} and $U_0$ from \cref{sec:coupling} and set $\hat U_0=\{0,-1\}^d\setminus\{0\}$, so that $\cU_0=\{\hat U_0\}$. We seek to control the $\{U_0\}$-CP on the box $\Lambda=\{1,\dots,n\}^d$ with parameter $q_0\in[0,1)$ in terms of the $\hat U_0$-LPP slowed down by some factor. As noted in \cref{subsec:LPP}, the $\hat U_0$-LPP coincides with the usual $\{-e_1,\dots,-e_d\}$-LPP, so we will be able to exploit the fact that this model is known to propagate ballistically.

Recall the boxes $\hat B_x$ from \cref{sec:renorm} (also recall \cref{fig:renorm}) and the direction $u\in S^{d-1}$ such that $U_0\subset \bbH_u$. For each $x\in\bbZ^d$ we will define its \emph{passage time} $t_x\in[0,\infty)$. Let
\begin{equation}
\label{eq:S}
\hat\Lambda=\left\{x\in\bbZ^d:\hat B_x\cap \Lambda\ne\varnothing\right\}
\end{equation}
denote the renormalised version of $\Lambda$. We start by recalling a result of \cite{Greenberg09}.

\begin{proposition}
\label{prop:LPP:filling} Denote by $(s_x)_{x\in \mathbb Z^d}$ the passage times of the standard $\{-e_1,\dots,-e_d\}$-LPP in the set $\hat\Lambda$ defined in \cref{eq:S}, with $s_x=0$ if $x\not\in \hat\Lambda$. Then, there exists $C=C(U_0)<\infty$ such that
    for every $\delta>0$, for $n$ large enough (depending on $\delta$) we have
    \begin{equation}
        \label{sxnC}
    \mathbb P\left(\max_{x\in\bbZ^d} s_x\ge n C\right)\le \delta.
    \end{equation}
\end{proposition}
\begin{remark}
    The proposition is, essentially,  a byproduct of the main result of \cite{Greenberg09}. In that paper, the authors upper bound the mixing time of a discrete-time Markov chain on $d-$dimensional discrete monotone sets.  That dynamics depends on a bias parameter $\lambda$, that is assumed to be sufficiently large, depending only on the dimension $d$. In the limit $\lambda\to\infty$, the dynamics reduces to (a discrete-time version of) standard $d-$dimensional LPP. 
For the reader's convenience, we give a streamlined proof of \cref{prop:LPP:filling} in the context of (continuous-time) LPP and along the way we slightly improve the main statement of \cite{Greenberg09} (which, translated into our language, would give \cref{sxnC} with $nC$ replaced by $n C \log (1/\delta)$.)
\end{remark}
\begin{proof}[Proof of \cref{prop:LPP:filling}]
First of all, using monotonicity of LPP, we can replace the set $\hat\Lambda$ by a cube $\{1,\dots,\ell\}^d$ that contains it. 
Using the fact that $\hat\Lambda$ has diameter upper bounded by $n$ times a $U_0$-dependent constant, we assume henceforth that $\hat\Lambda$ is a cube with $\ell=O(n)$. 

We say that a subset $\sigma\subset \hat\Lambda$ is a \emph{monotone set} if the conditions that $x\in \sigma$ and $y\preceq x$ (that is, $y\le x$ componentwise) imply that $y\in \sigma$. We define a continuous-time Markov chain on the collection $\Sigma$ of monotone sets; the state of the chain at time $t$ is denoted $\sigma(t),$ and we let $\sigma(0):=\varnothing$. Each $x\in \hat\Lambda$ has an independent Poisson clock of rate $1$. When the clock at $x$ rings,  if $x\not\in \sigma(t-)$ and if $\sigma(t-)\cup\{x\}\in\Sigma$, then we let $\sigma(t)=\sigma(t-)\cup\{x\}.$ Otherwise, nothing happens. It is easy to check that this dynamics is equivalent to standard  $\{-e_1,\dots,-e_d\}$-LPP in the sense that the process $(\sigma(t))_{t\ge0}$
has the same law as the process $(\{x\in\hat\Lambda:s_x\le t\})_{t\ge0}.$ In particular, \cref{sxnC} is equivalent to proving that
\begin{equation}
\label{eq:riformulata}
\mathbb P\left(\sigma(n C)=\hat\Lambda\right)>1-\delta.
\end{equation}

The trivial measure concentrated on the absorbing state $\sigma=\hat\Lambda$ is stationary for the monotone set dynamics, because $\sigma(t)\supset \sigma(s)$ for $t>s$. For this reason, \cref{eq:riformulata} is equivalent to the fact that the mixing time of the monotone set dynamics satisfies 
$t_{\rm mix}(\delta)\le Cn$, for $n$ large enough. The idea of \cite{Greenberg09} is to use path coupling with an exponential metric. That is, fix $\gamma>0$ and let $\underline 1=(1,\dots,1)\in\bbZ^d$. Given $\sigma,\sigma'\in\Sigma$ that differ by a single vertex, say, $\sigma=\sigma'\cup\{x\}$, define $d_\gamma(\sigma,\sigma')=e^{-\gamma \<x,\underline 1\>}$,  and extend $d_\gamma$ to be a distance on the whole state space $\Sigma$. The path coupling method (see \cite{Greenberg09} and \cite{Levin09}*{Ch.\ 14} in the discrete time setting, but the method is analogous in continuous time) consists in proving that there exists some $\alpha>0$ such that for each pair $(\sigma,\sigma')$ differing by a single vertex there exists a coupling of the processes $(\sigma(t),\sigma'(t))$ with initial conditions $\sigma,\sigma'$ such that
\begin{equation}
\label{eq:derivative}
    \left.\frac d{dt}\mathbb E \left[d_\gamma(\sigma(t),\sigma'(t))\right]\right|_{t=0}\le -\alpha d_\gamma(\sigma,\sigma').
\end{equation}
If this is the case, then
 \begin{equation}
 \label{SK}
 t_{\rm mix}(\delta)\le \frac1\alpha \log({\rm diam}_\gamma(\Sigma)/\delta),    
 \end{equation}
 with ${\rm diam}_\gamma(\Sigma)$ the diameter of $\Sigma$ with respect to the metric $d_\gamma$. In the case of the exponential metric above,  $\log {\rm diam}_\gamma(\Sigma)=O(\ell)$ and, recalling that $\ell=O(n),$  \cref{SK} implies that $t_{\rm mix}(\delta)\le Cn$, for $n$ large enough, as desired. The computation of the time derivative in \cref{eq:derivative} is very similar to the computation of the discrete-time derivative in \cite{Greenberg09}*{Sec.\ 4}, so we do not repeat it: one finds that there is exactly one update (the one at $x$) that decreases the distance by $d_\gamma(\sigma,\sigma')$, and at most $d$ updates (at vertices neighbouring $x$) that increase the distance by $e^{-\gamma}d_\gamma(\sigma,\sigma')$. Altogether,
 \begin{equation}
    \left.\frac d{dt}\mathbb E\left[ d_\gamma\left(\sigma(t),\sigma'(t)\right)\right]\right|_{t=0}\le -\left(1-d e^{-\gamma}\right)d_\gamma(\sigma,\sigma'),
\end{equation}
which gives \cref{eq:derivative} with $\alpha>0$ for $\gamma$ large enough. This concludes the proof.
\end{proof}

Recall the clock rings $(P_x)_{x\in\bbZ^d}$ and uniform random variables $(\Upsilon_x(t))_{t\in P_x}$ from \cref{subsec:KCM}. We set $t_x=0$ for all $x\in \bbZ^d\setminus\hat\Lambda$. For $x\in\hat\Lambda$, we define $t_x\ge 0$ by induction as follows, assuming that $t_{x+y}$ is already defined for all $y\in\hat U_0$. The stopping time $t_x$ is the first time $t>\tilde t_x:=\max_{y\in\hat U_0}t_{x+y}$ such that for every $z\in\hat B_x\cap \Lambda$ there exists $p_z\in P_z\cap [\tilde t_x,t]$ such that: \begin{itemize}
    \item $\Upsilon_z(p_z)> q_0$,
    \item the collection $(p_z)_{z\in \hat B_x}$ satisfies that $\<z,u\>>\<z',u\>$ implies $p_z>p_{z'}$. 
\end{itemize}
In words, once the boxes $\hat B_{x+y}$ for $y\in\hat U_0$ have been treated, we require the occurrence of a sequence of clock rings corresponding to healing in the $\{U_0\}$-CP and occurring in the order of increasing scalar product with the direction $u$.

The use of these passage times is clear in the following lemma, where, as usual, $\zeta$ denotes the $\{U_0\}$-CP on $\Lambda$ with parameter $q_0$ and boundary condition $\bone_{\bbZ^d\setminus\Lambda}$.
\begin{lemma}
\label{lem:CP-LPP-comparison}
For any $x\in\bbZ^d$, $z\in \hat B_x$ and $t\ge t_x$ we have $\zeta^\bone_z(t)=\zeta^\bzero_z(t)$. 
\end{lemma}
\begin{proof} If $z\in \hat B_x\setminus \Lambda$, the claim is trivial because $\zeta_{{z}}^\bone(t)=\zeta_{{z}}^\bzero(t)=1$ for all $t$. Otherwise, we proceed by induction and we assume that the claim has been proven for all $z'\in\hat B_y$, $y\in x+\hat U_0$, and we want to prove it for $z\in\hat B_x$.
By the induction hypothesis, the two processes are perfectly coupled for $t\ge \tilde t_x$ in the boxes $\hat B_y$ for $y\in x+\hat U_0$. 

For $z,z'\in \bbZ^d$ we write $z\prec z'$ if $
\langle z,u\rangle< \langle z',u\rangle$. Within the box $\hat B_x$, we proceed by induction with respect to this partial order to show that $\zeta^\bone_z(t)=\zeta^\bzero_z(t)$ for all $t\ge p_z$, where $(p_z)_{z\in\hat B_x}$ are provided by the definition of $t_x$. Assume that this is proved for all $z'{\prec} z\in\hat B_x$. Then $\zeta^\bone_z(p_z)=\zeta^\bzero_z(p_z)=0$, since $\Upsilon_z(p_z)>q_0$ (recall \cref{eq:def:CP}). Moreover, using \cref{choice-of-geometry} and the fact that $U_0{\subset}\mathbb H_u$, we obtain 
\[z+U_0\subset\bigcup_{y\in\{0,-1\}^d}\hat B_{x+y}\cap \left\{z'\in\bbZ^d:z'\prec z\right\}.\]
Since the r.h.s.\ above is coupled at all times $t\ge p_z$ by induction hypothesis, the definition of the $\{U_0\}$-CP (see \cref{eq:def:CP,eq:def:cx}) gives that $\zeta^\bone_z(t)=\zeta^\bzero_z(t)$ for all $t\ge p_z$, as desired.
\end{proof}
Combining \cref{prop:LPP:filling,lem:CP-LPP-comparison}, we can now prove the following.
\begin{proposition}
    \label{prop:CP:mixing}
  Let $q_0\in[0,1)$. There exists $c_0=c_0(U_0,q_0)<\infty$ such that for every $\delta>0$ and $n$ large enough (depending on $\delta$), the mixing time of the $\{U_0\}$-CP with parameter $q_0$ (in the box $\Lambda$, with ${\bf 1}$ boundary condition) satisfies
    \begin{equation}
    \label{cn}
        \tmix(\delta)\le c_0n.
    \end{equation}
    \end{proposition}
Note that in \cref{prop:CP:mixing} we do not need $q_0$ to be large. However, for small $q_0$ the CP in infinite volume has trivial upper invariant measure and the mixing time in the box should be logarithmic in $n$.
\begin{proof}
It is not hard to check that $t_x-\tilde t_{x}$ is stochastically dominated by the sum of $|\hat B|$ independent exponential random variables of parameter $1-q_0$. On the other hand, the size of $\hat B$ depends only on the set $U_0$. Since the sum of $N$ exponential random variables is stochastically dominated by a single exponential random variable with suitably large ($N$-dependent) expectation, we have that $t_x-\tilde t_x$ is stochastically dominated by an exponential random variable of parameter depending on $q_0$ and $U_0$. Since LPP is monotone in the quantities $t_x-\tilde t_x$, which are clearly independent for different $x$, we get that $t_x$ is smaller than the passage time of a $\hat U_0$-LPP slowed down by a factor $\theta(q_0,U_0)\in[1,\infty)$, which we noted in \cref{subsec:LPP} to coincide with the standard $\{-e_1,\dots,-e_d\}$-LPP slowed down by the same factor. 

By \cref{lem:CP-LPP-comparison} and attractiveness of CP we have that $\max_{x\in\bbZ^d} t_x$ is an upper bound on the coupling time of the $\{U_0\}$-CP on $\Lambda$ with boundary condition $\bone$. The proof is concluded by \cref{prop:LPP:filling,eq:def:tmix}, as in \cref{eq:convergence:coupling}.
\end{proof}

\section{Assembling Theorem~\ref{th:mixing}}
\label{sec:mixing}
\begin{proof}[Proof of the lower bound in \cref{eq:UBLBtmix}]
Let 
\[ \Lambda_\ell=\left\{\lfloor n/2\rfloor -\ell,\dots,\lfloor n/2\rfloor +\ell\right\}^d.\]
One can choose $\ell $ large enough (depending  on $\delta$ and $q$, but not on $n$) so that the probability, under the stationary measure for the process in the box $\Lambda=\{1,\dots,n\}^d$ with boundary condition $\bone$, that  $\Lambda_\ell$ is in state $\bzero$ is smaller than $\delta/2$. For the $\mathcal U-$KCM, this is obvious, since the stationary measure is a Bernoulli product measure with parameter $q>0$. For the $\mathcal U-$CP, this follows from the fact that the stationary measure is stochastically larger than the restriction to $\Lambda$ of the upper invariant measure $\bar\nu$ on $\bbZ^d$ (recall \cref{subsec:CP}). The latter is  ergodic (see \cite{Liggett05}*{Theorem III.2.3.(f)}) and non-trivial under the assumption that $q{>\qc[CP]}$ (recall \cref{eq:def:qcCP}). But $q>\qc[CP]$, since $q$ is sufficiently close to $1$ and $\qc[CP]<1$ for $\cU$ which is not trivial subcritical (recall \cref{rem:qcCP}).

On the other hand, starting the dynamics (either KCM or CP) from the $\bzero$ configuration in $\Lambda$ with $\bone$ boundary condition, with high probability it takes a time 
at least $c n$ (with $c$ independent of $\delta$) before the state in $\Lambda_\ell$ changes. This follows from the definition of the dynamics, the fact that the box $\Lambda_\ell$ is at distance at least $n/4$ from the boundary of $\Lambda$ and from standard estimates on first passage percolation (see e.g.\ \cite{Liggett99}*{Section I.1}).
\end{proof}

\begin{proof}[Proof of the upper bound in \cref{eq:UBLBtmix}]
We need  to show that, for every $\delta>0$ and $n $ large enough (depending on $\delta$), the $\mathcal U$-KCM in $\Lambda$ with parameter $q$ and boundary condition ${\bf 1}$, started from the initial state ${\bf 1}$  has coupled at time $cn$ with the one started from an arbitrary configuration $\xi'\in\Omega_\Lambda$, with probability at least $1-\delta$. The proof for the $\cU$-CP is identical and therefore omitted.

From \cref{prop:CP:mixing} we know that (up to a total variation error $\delta/2$) for any $t\ge T_0:=c_0n$ the $\{U_0\}$-CP with parameter $q_0\le q$ large enough, in $\Lambda$ with boundary condition $\bone$ and any initial condition, has coupled. Namely, denoting these processes by $\zeta^{\xi'',\Lambda}$ for initial conditions $\xi''\in\Omega_\Lambda$, we get
\[\bbP\left(\forall t\ge T_0,\zeta^{\bzero,\Lambda}(t)=\zeta^{\bone,\Lambda}(t)\right)\ge 1-\delta/2.\]
By attractiveness we have that at any time $\zeta^{\bone,\Lambda}$ dominates the restriction to $\Lambda$ of the infinite volume $\{U_0\}$-CP $\zeta^\bone$ with initial condition $\bone$ and parameter $q_0$. Thus,
\begin{equation}
\label{eq:zeta10}
\bbP\left(\forall t\ge T_0,\forall x\in\Lambda,\zeta_x^\bone{(t)}\le\zeta_x^{\bzero,\Lambda}{(t)}\right)\ge 1-\delta/2.\end{equation}

For all $t\ge T_0$, let $O_t^\bone$ denote the set of orange healthy sites for $\zeta^\bone$ and let $O_t^{\bzero,\Lambda}$ denote the one of $\zeta^{\bzero,\Lambda}$, where both are initialised at time $T_0$ to be equal to the set of all healthy sites and then defined via \cref{eq:def:Ot}. Using \cref{eq:def:Ot}, is not hard to check by induction on the number of clock rings in $\Lambda$ in the time interval $[T_0,t]$, that, if the event in the left hand side of \cref{eq:zeta10} occurs, then for any $t\ge T_0$ we have $O^\bone_t\cap\Lambda\supset O_t^{\bzero,\Lambda}$. Hence,
\[\bbP\left(\forall t\ge T_0, O^\bone_t\cap\Lambda\supset O_t^{\bzero,\Lambda}\right)\ge 1-\delta/2.\]
Moreover, applying \cref{lem:coupling} (and \cref{claim:comparison} up to time $T_0$), a union bound and translation invariance, we have that for $t\ge T_0$
\begin{equation}
\label{eq:deltamezzi}
\mathbb P\left(\eta^{\bf 1}(t)\ne \eta^{\xi'}(t)\right)\le \bbP\left(O^{\bzero,\Lambda}_t\neq\varnothing\right)\le \delta/2+|\Lambda|\cdot \bbP\left(0\in O_t^\bone\right).
\end{equation}
The proof of \cref{th:mixing} is therefore concluded, using monotonicity once again together with \cref{eq:decay:Ot}, taking e.g.\ $t=T_0+\sqrt{n}$, since none of the other quantities depend on $n$. Indeed, in \cref{sec:convergence}, \cref{eq:decay:Ot} was valid for the $\{U_0\}$-CP starting from an initial configuration renormalising to $\xi\sim\mu_{1-\varepsilon_0}$ with $\varepsilon_0$ small enough, while here we simply have $\varepsilon_0=0$.
\end{proof}

\section*{Acknowledgements}\
This work was supported by the Austrian Science Fund (FWF): P35428-N.  We thank Paul Chleboun, Damiano de Gaspari, Markus Heydenreich, Laure Mar\^ech\'e, Rob Morris, R\'eka Szab\'o and Cristina Toninelli for enlightening discussions and to the anonymous referees for valuable comments.

\bibliographystyle{plain}
\bibliography{Bib}

\end{document}